\documentclass[12pt,a4paper,english]{smfart}
\usepackage[english]{babel}
\usepackage{ragged2e}
\usepackage{smfthm, mathabx}
\usepackage{stmaryrd}
\usepackage{amssymb}
\usepackage{amsmath}
\usepackage{amsthm}
\usepackage{amsfonts}
\usepackage{graphicx}
\usepackage{enumerate}
\usepackage{comment}
\usepackage{booktabs}
\usepackage{float}

\usepackage{tikz-cd}
\usepackage{calc} % for width calculations

\usepackage{appendix}

\usepackage{euscript,mathrsfs}
\usepackage[utf8]{inputenc} 
\usepackage{longtable}
\usepackage{dsfont}

\usepackage[OT2,T1]{fontenc}

\usepackage[all,cmtip]{xy}
\usepackage{alltt}

\usepackage{calrsfs}

\makeatletter
\def\thm@space@setup{%
  \thm@preskip=\parskip \thm@postskip=0pt
}
\makeatother

\xyoption{all}

\usepackage{float}

\usepackage{fullpage}

\usepackage{color}

\newtheorem{coco}{Corollary}
\newtheorem{Theorem}{Theorem}
\newtheorem{prp}{Proposition}
\newtheorem{Exa}{Example}

\author[Oussama Hamza]{Oussama Hamza}
\address{Institute for Advanced Studies in Mathematics, Harbin Institute of Technology, Harbin, China 150001}
\email{ohamza3@uwo.ca}

\title{Filtrations and cohomology on graph products}

\thanks{The author thanks Thomas Weigel for discussions on graphs, group filtrations, his interest on this work and several references. He is grateful to Henrique Souza for discussions on strict filtered maps. He also thanks Li Cai for several comments. The author acknowledges Donghyeok Lim for several discussions, readings, and his interest in this work. He is thankful to Christian Maire, J{\'a}n Min{\'a}{\v c}, Nguy{\^e}n Duy T{\^a}n, Jun Wang, Heon Lee, Yuanyuan Jing, Zhizheng Yu and Keping Huang for discussions and their support. He is also grateful to Elyes Boughattas, Baptiste Cercle, Sayantan Roy Chowdhury, Tao Gong, Vladimir Gorchakov, Olivier Wittenberg, Prajwal Udanshive, Matthias Franz, Larry So, Steven Amelotte, Karl Lorensen and Taras Panov for their interests in this work.}

\newcommand{\Q}{\mathbb{Q}}
\newcommand{\F}{\mathbb{F}}
\newcommand{\Z}{\mathbb{Z}}
\newcommand{\NN}{\mathbb{N}}

\def\N{{\rm N}}

\def\P{\mathbb{P}}

\def\log{{\rm log}}

\def\grad{{\rm Grad}}

  \DeclareFontFamily{U}{wncy}{}
    \DeclareFontShape{U}{wncy}{m}{n}{<->wncyr10}{}
    \DeclareSymbolFont{mcy}{U}{wncy}{m}{n}
    \DeclareMathSymbol{\Sh}{\mathord}{mcy}{"58} 

\def\bX{{\mathbf{X}}}
\def\bE{{\mathbf{E}}}

\def\O{{\mathcal O}}
\def\Rr{{\mathcal R}}
\def\PP{{\mathcal P}}
\def\J{{\mathcal J}}
\def\E{{\mathcal E}}
\def\Ll{{\mathcal L}}
\def\U{{\mathcal U}}
\def\A{{\mathcal A}}
\def\AA{{\mathbb A}}

\def\rk{{\rm rank}}

\def\I{{\mathcal I}}

\def\fq{{\mathbb F}}
\def\L{{\rm L}}

\begin{document}

\begin{abstract}
Let~$p$ be a prime. We resolve a question posed by Min{\'a}{\v c}–Rogelstad–Tân. We relate the Zassenhaus and the lower central series of pro-$p$ groups under a torsion-freeness condition.

We also study graph products of (pro-$p$) groups under natural assumptions. In particular, we compute their graded Lie algebras associated with the previous filtrations, as well as their cohomology over~$\mathbb{F}_p$. Our approach relies on various filtrations of amalgamated products, as studied in Leoni's PhD thesis. Explicit examples are provided using the Koszul property. 

As a concrete application, we compute the cohomology over~$\mathbb{F}_p$ and the graded Lie algebras associated with the filtrations of graph products of fundamental groups of surfaces. These groups furnish new examples satisfying the torsion-freeness condition, which arises in the question of Mináč–Rogelstad–Tân.
%and applications for finitely generated abstract groups.
\end{abstract}

\maketitle

Right Angled Artin group (RAAG) theory associates a graph to a group. It plays a fundamental role in geometric group theory and combinatorics. For further details, consult~\cite{charneysurvey} and~\cite{bartholdi2020right}. In recent years, this theory also developed several applications for absolute Galois groups. Snopce and Zalesskii~\cite{snopce2022right} have precisely determined which RAAGs are realized as maximal pro-$p$ quotients of absolute Galois groups.
%in analogy with a result of Droms~\cite{droms1987subgroups}. 
These results were later extended by Blumer, Quadrelli and Weigel \cite{blumer2023oriented}. Recently the author~\cite{hamza2025pythagorean} together with Maire, Min{\'a}{\v c} and Tân used RAAG theory to give presentations of maximal pro-$2$ quotients of absolute Galois groups of Formally Real Pythagorean fields of finite type. This result yielded new examples of pro-$2$ groups which are not maximal pro-$2$ quotients of absolute Galois groups.
%and found the first example of a pro-$2$ group which is not a pro-$2$ absolute Galois groups but checks Massey Vanishing, Kernel Unipotent and Koszul properties which were originally conjectured to infer absolute Galois groups. 

Graph products, originally introduced in \cite{ruthgraph}, generalize RAAGs. With their filtrations and cohomology rings, they play a fundamental role in geometric group theory. We refer to~\cite{panovpseries} and~\cite[Chapter~$4$]{panovbook2015}. Concretely, we fix an undirected graph~$\Gamma\coloneq (\bX,\bE)$ on~$\bX\coloneq \{1,\dots ,k\}$ and a family of (pro-$p$) groups $G\coloneq \{{}_1G,\dots,{}_kG\}$. Note that the category of (pro-$p$) groups is endowed with a coproduct $\coprod$. Then, we define the graph product~$\Gamma G$ as the quotient of~$\coprod_{i=1}^k {}_iG$ by the normal (closed) subgroup generated by~$\lbrack {}_iG,{}_jG \rbrack$ where $\{i,j\}\in \bE$. 

In this paper, we study filtrations and cohomology groups of graph products under a torsion-freeness condition. This condition is motivated by a question of Mináč–Rogelstad–Tân \cite[Question~$2.13$]{MRT}, which we answer fully and affirmatively. The same torsion-freeness condition was also studied independently by Marmo–Riley–Weigel \cite{WeigelRiley}, where it is referred to as~$\gamma$-free. More precisely, \cite[Theorem B]{WeigelRiley} relates the Zassenhaus and lower central filtrations under this condition (see also Remark~\ref{Marmo-Riley-Weigel}). At present, few examples satisfying this torsion-freeness condition are known. In this work, we show that graph products provide several new examples derived from previously known ones.
%of pro-$p$ groups that positively answer~\cite[Question~$2.13$]{MRT}, stated by~Min{\'a}{\v c}-Rogelstad-Tân. More generally, we fully resolve this question.

\subsection*{Filtrations on groups and the question of~Min{\'a}{\v c}-Rogelstad-Tân} In general, we denote by $Q$ a finitely generated pro-$p$ group. Let~$\AA$ be either $\F_p$, the field with $p$ elements, or $\Z_p$, the ring of~$p$-adic integers. If~$A$ and $B$ are two subgroups of~$Q$, we define~$\lbrack A, B\rbrack$ the closed normal subgroup of~$Q$ generated by~$\lbrack a,b\rbrack\coloneq a^{-1}b^{-1}ab$. We also define~$AB$ the closed normal subgroup of~$Q$ generated by~$ab$ with $a\in A$ and $b\in B$. Consider~$Q_1(\AA):=Q$, and for every integer~$n\geq 2$:
\begin{equation}\label{def filt}
Q_n(\Z_p):=[Q,Q_{n-1}(\Z_p)], \quad \text{and} \quad Q_n(\F_p):=Q_{\lceil \frac{n}{p} \rceil}(\F_p)^p\prod_{i+j=n}[Q_i(\F_p),Q_j(\F_p)],
\end{equation}
where~$\lceil \frac{n}{p} \rceil$ is the smallest integer greater than or equal~$\frac{n}{p}$. The filtration $Q_\bullet(\F_p)$ is called the Zassenhaus filtration of $Q$, while $Q_\bullet(\Z_p)$ is the (topological) lower central series of~$Q$. For every positive integer~$n$, these two filtrations are related by the Lazard formula: 
$$Q_n(\F_p)=\prod_{ip^j\geq n}Q_i(\Z_p)^{p^j}.$$
See for instance~\cite[Theorem $11.2$]{DDMS} for a proof.

From these filtrations, we construct the $\AA$-graded (locally finite) Lie algebras:
\begin{multline*}
\Ll(\AA,Q):=\bigoplus_n \Ll_n(\AA,Q), \quad \text{where } \\\Ll_n(\AA,Q):=Q_n(\AA)/Q_{n+1}(\AA), \text{ and } \quad a_n(\AA,Q):=\rk_{\AA} \Ll_n(\AA,Q).
\end{multline*}
In this paper, we are mostly interested in groups $Q$ such that $\Ll(\AA, Q)$ is torsion-free over $\AA$. Although this condition is automatic when $\AA:=\F_p$, it is generally difficult to verify when $\AA:=\Z_p$. Our main example follows from Labute \cite[Theorem]{labutedescending} (see also \cite[Remark $2.12$]{MRT}). We also refer to~\cite{WeigelRiley} for several other examples. We consider~$\widehat{S_g}^p$, the pro-$p$ completion of the fundamental group of a surface of genus $g$, given by the presentation:
\begin{equation}\tag{$\P_g$}\label{example pres}
 \langle x_1,\dots x_g, y_1, \dots, y_g \mid \prod_{j=1}^{g} \lbrack x_j,y_j\rbrack=1 \rangle.
 \end{equation}
%Several other new examples will be given using our results on graph products.

Similarly to~\cite[Theorem]{labutedescending}, we have an explicit formula for $a_n(\AA,\widehat{S_g}^p)$ depending only on~$p$,~$n$ and~$g$. Let~$n$ be a positive integer, and let us write~$n\coloneq mp^{\nu_p(n)}$, where~$\nu_p(n)$ is the $p$-adic valuation of $n$, and $m$ and $p$ are coprime. As a consequence of~\cite[Lemma~$5.4$ and Theorem~$2.9$]{MRT}, we have the relation:
\begin{equation*}
a_n(\F_p,\widehat{S_g}^p)=a_m(\Z_p,\widehat{S_g}^p)+a_{mp}(\Z_p,\widehat{S_g}^p)+a_{mp^2}(\Z_p,\widehat{S_g}^p)+\dots + a_n(\Z_p,\widehat{S_g}^p).
\end{equation*}

Assume that~$Q$ is a finitely generated pro-$p$ group such that~$\Ll(\Z_p,Q)$ is torsion-free. Min{\'a}{\v c}-Rogelstad-T{\^ a}n \cite[Question $2.13$]{MRT} asked whether the above equality holds for~$Q$. In~\cite[Theorem $3.5$]{HAMZA2023172}, the author gave a partial answer. In this paper, we give a complete answer:

\begin{Theorem}\label{answerminac}
Assume that $\Ll(\Z_p,Q)$ is torsion-free. Let~$n$ be a positive integer, and let us write~$n\coloneq mp^{\nu_p(n)}$, with $m$ coprime to $p$. Then
\begin{equation}\label{equation minac}
a_n(\F_p,Q)=a_m(\Z_p,Q)+a_{mp}(\Z_p,Q)+a_{mp^2}(\Z_p,Q)+\dots + a_n(\Z_p,Q).
\end{equation}
\end{Theorem}

The proof of Theorem~\ref{answerminac} is presented in the first section. It is based on Theorem~\ref{answerminac2}, which relies on strict filtered maps. These maps were well studied by Leoni~\cite{leoni2024zassenhaus}. Let us note that Marmo, Riley and Weigel~\cite[Theorem~$B$]{WeigelRiley} allows us to deduce an alternative proof of Theorem~\ref{answerminac2} (see Remark~\ref{Marmo-Riley-Weigel}).

\subsection*{Results on filtrations of graph products} We now use graph product theory to obtain new examples of pro-$p$ groups $Q$ such that~$\Ll(\Z_p,Q)$ is torsion-free. 

From a collection $G:=\{{}_1G,\dots, {}_kG\}$ of finitely generated pro-$p$ groups, we obtain a collection of~$\AA$-graded Lie algebras~$\Ll(\AA,G):=\{\Ll(\AA, {}_1G), \dots, \Ll(\AA, {}_k G)\}$. In this paper, we consider families $G$ satisfying the property:
\begin{equation}\tag{$\ast_{\AA}$}\label{condA} 
\text{For every $1\leq i \leq k$, the algebra $\Ll(\AA, {}_iG)$ is torsion-free over $\AA$.}
\end{equation}
This condition is automatic over~$\F_p$, it is introduced only to standardize the notations.
Since the category of graded algebras also admits a coproduct (that we also denote by~$\coprod$), we define the graph product:
\begin{equation*}
\begin{aligned}
 \Gamma \Ll(\AA,G)&\coloneq \coprod_{i=1}^k \Ll(\AA,{}_iG)/ \langle [\Ll(\AA,{}_uG),\Ll(\AA,{}_vG)]|\quad \{u,v\}\in \bE \rangle.
\end{aligned}
\end{equation*}
This paper is devoted to the study of filtrations on $\Gamma G$. Precisely, we show:

\begin{Theorem}\label{main result}
Let $G$ be a family of finitely generated pro-$p$ groups. Assume that $G$ satisfies \eqref{condA}, then $(i)$ the algebra $\Ll(\AA, \Gamma G)$ is torsion-free, and
%$(ii)$ $\gamma(\AA,\Gamma G)$ is the fastest $\AA$-filtration of $\Gamma G$, i.e. we have:
%$$\gamma_n(\AA, \Gamma G)=(\Gamma G)_n(\AA):=\{g\in \Gamma G, g-1\in E_n(\AA,\Gamma G)\},$$
$(ii)$ we obtain the equality
$$\Ll(\AA, \Gamma G)\simeq \Gamma \Ll(\AA,G).$$ 
\end{Theorem}

The proof of Theorem~\ref{main result} is presented in Section~$2$. The main argument follows from filtrations on amalgamated products studied in Leoni's thesis~\cite[Chapter~$4$]{leoni2024zassenhaus}. This is precisely Theorem~\ref{leoni result} of this paper. We refer to~\cite[§9.2]{RibesZal} for further details on amalgamated products. %Precisely, assume that~$\Gamma$ is not a complete graph. Then there exists three subgraphs~${}^a \Gamma$, ${}^b\Gamma$ and~${}^c\Gamma$ of~$\Gamma$ and three subfamily~${}^a G$, ${}^b G$ and~${}^cG$ of~$G$ such that
%$$\Gamma G\simeq {}^a\Gamma {}^a G\coprod_{{}^c \Gamma {}^c G}{}^b \Gamma {}^b G.$$ 
%The previous isomorphism was also observed by Green~\cite[Lemma~$3.20$]{ruthgraph}.

\subsection*{Koszulity and gocha series of graph products}
A graded algebra over~$\F_p$ is said to be Koszul if it is generated in degree~$1$, with relations of degree~$2$ and admits a linear resolution. The Koszul property, also called Koszulity, allows us to efficiently compute the Zassenhaus filtration and the dimensions of the cohomology groups over~$\F_p$. This property has significant applications in current algebra and geometry. For further details, we refer for instance to~\cite{polishchuk2005quadratic}, \cite{weigel2015graded} and~\cite{weigel652koszul}.
%We relate Theorem~\ref{answerminac} and Theorem \ref{main result} to compute explicitly $a_n(\Z_p, \Gamma G)$.

Consider $E(\AA,Q)$ the completed group algebra of $Q$ over~$\AA$. This algebra is endowed with the filtration~$\{E_n(\AA,Q)\}_{n\in \NN}$ given by the $n$-th power of the augmentation ideal. This allows us to introduce the (locally finite) graded-$\AA$-algebra
$$\E(\AA,Q)\coloneq \grad(E(\AA,Q))\coloneq \bigoplus_{n\in \NN} \E_n(\AA,Q), \quad \text{where } \E_n(\AA,Q)\coloneq E_n(\AA,Q)/E_{n+1}(\AA,Q).$$
%By Magnus isomorphism, this algebra is a quotient of $\E$ the algebra of noncommutative polynomials over $\AA$ on $X_1,\dots, X_d$, where $d$ is the minimal number of generators of $Q$.
We also define
\begin{equation*}
\begin{aligned}
gocha(\AA, Q,t)&\coloneq 1+\sum_{n\geq 1} c_n(\AA,Q)t^n, \quad \text{ where } c_n(\AA,Q)\coloneq \dim_{\F_p}\E_n(\AA,Q),
\\ \log(gocha(\AA,Q,t))&\coloneq -\sum_{n\geq 1}\frac{(gocha(\AA,Q,t)-1)^n}{n}\coloneq \sum_{n\in \NN} b_n(\AA,Q)t^n, \quad  \text{ where } b_n(\AA,Q) \in \Q.
\end{aligned}
\end{equation*}

In this paper, computations of gocha series are justified by the following explicit relations. They come from \cite[Theorem~$2.9$]{MRT}. 
%and Equality~\eqref{equation minac}. 
Let $\mu$ be the M{\"o}bius function and we write $n\coloneq mp^{\nu_p(n)}$ where~$m$ and~$n$ are coprime. We have
\begin{equation}\label{equation minac2}
\begin{aligned}
a_n(\Z_p,Q)=& \frac{1}{n}\times \sum_{m|n}\mu(n/m) mb_m(\Z_p,Q), 
\\a_n(\F_p,Q)=& \frac{1}{n}\times \sum_{v=0}^{\nu_p(n)}\sum_{p^vm|n}\mu\left(\frac{n}{p^vm}\right)p^vmb_m(\F_p,Q).
\end{aligned}
\end{equation}
%a_n(\Z_p,Q)=\frac{1}{n}\times \sum_{v=0}^{\nu_p(n)}\sum_{p^vm|n}\mu\left(\frac{n}{p^vm}\right)p^vmb_m(Q).
Theorem~\ref{answerminac} is a consequence of Theorem~\ref{answerminac2}: for every positive integer~$n$, we have~$c_n(\F_p,Q)=c_n(\Z_p,Q)$ (and so~$b_n(\F_p,Q)=b_n(\Z_p,Q)$). Precisely, Equalities~\eqref{equation minac2} and Theorem~\ref{answerminac2} allow us to recover the Equality~\eqref{equation minac}.

Let~$H^\bullet(Q)$ be the graded (continuous) cohomology algebra of~$Q$ over~$\F_p$.
When~$\E(\F_p,Q)$ is Koszul, it was shown in~\cite[Proposition $1$]{Hamza25} (see also~\cite{leoni2024zassenhaus}) that~$H^\bullet(Q)$ is the quadratic dual of~$\E(\F_p,Q)$. So~$H^\bullet(Q)$ is Koszul.  We denote by~$\E$ the algebra of noncommutative polynomials over $\{X_1,\dots, X_d\}$, with~$d$ the minimal number of generators of $Q$. We write~$H^\bullet(Q)\coloneq \E/\I^!(Q)$, where~$\I^!(Q)$ is a two-sided ideal generated by a family~$S^!(Q)$ of homogeneous polynomials of degree~$2$. 
%a family in $\E_2$ and $\I^!(Q)$, the two sided ideal of $\E_2$ generated by $S^!(Q)$ such that $H^\bullet(Q):=\E/\I^!(Q)$.  
We define~$h^n(Q)\coloneq \dim_{\F_p}H^n(Q)$, which is well defined when $\E(\F_p, Q)$ is Koszul. 
We infer:

\begin{prp}[Theorem~\ref{Koszul graph}, Proposition~\ref{computation gocha series} and Corollaries~\ref{coho graph} and~\ref{app Koszul}] \label{computekosz}
Assume that $\E(\F_p,{}_iG)$ is Koszul, for $1\leq i \leq k$, with a minimal system of generators $\{{}_iX_j\}$ for $1\leq j \leq {}_id$. Then $\E(\F_p, \Gamma G)$ is Koszul, with a minimal system of generators~$\{{}_iX_j\}$ for~$1\leq i \leq k$ and~$1\leq j \leq {}_id$.

Let~$\E$ be the graded algebra of noncommutative series on $\{{}_iX_j\}$ where $1\leq i\leq k$ and $1\leq j\leq {}_id$.
%$d:={}_1d+\dots+ {}_kd$  
%for every $n$ the integer $h^n(\Gamma G)$ is defined. 
We have:
\begin{multline*}
H^\bullet(\Gamma G)\simeq \E/\I^!(\Gamma G), \quad \text{and}
\\ gocha(\Gamma G,t)=\frac{1}{H^\bullet(\Gamma G,-t)}, \quad \text{where } \quad H^\bullet(\Gamma G,t)\coloneq 1+\sum_{n\in \NN} h^n(\Gamma G)t^n.
\end{multline*}
Furthermore $\I^!(\Gamma G)$ is the two-sided ideal of $\E$ generated by the family:
\begin{multline*}
S^!(\Gamma G)= \bigcup_{i=1}^k S({}_iG)^! \cup S_\Gamma^{(1)}\cup S_\Gamma^{(2)}, \quad \text{where }
\\ S_\Gamma^{(1)}\coloneq \{({}_uX_a)({}_vX_b)+({}_vX_b)({}_uX_a) \text{ for } 1\leq u < v \leq k \text{ and } 1\leq a\leq {}_ud, 1\leq b\leq {}_vd \}
\\S_{\Gamma}^{(2)}\coloneq \{({}_uX_a)({}_vX_b), \text{ for } u\neq v, \{ u,v\} \notin \mathbf{E} \text{ and } 1\leq a\leq {}_ud, 1\leq b\leq {}_vd \}\subset \E_2.
\end{multline*}
%Here $(u,v)\not\in \bE$ means either $\{u,v\}\not\in \bE$ or $u>v$. And 
%We denote by~$S^!({}_iG)$ the elements of degree~$2$ which generate relations of the ring $H^\bullet({}_iG)$. 
Moreover: 
$$h^n(\Gamma G)=\sum_{m\geq 1}\sum_{ (i_1,\dots, i_m) \in \Gamma_m}\sum_{n_1+\dots+n_m=n} h^{n_1}({}_{i_1}G)\dots h^{n_m}({}_{i_m}G),$$
where $\Gamma_m$ is the set of $m$-cliques (complete subgraphs of $\Gamma$ with~$m$ vertices).%, ordered by~$i_1<i_2<\dots <i_m$).
\end{prp}

\subsection*{Application to abstract groups}
We obtain analogous results for abstract groups. Let $Q$ be a finitely generated abstract group. Assume that~$A$ and~$B$ are subgroups of~$Q$. We define~$AB$ the normal subgroup of~$Q$ generated by~$ab$, and~$\lbrack A, B\rbrack$ the normal subgroup of~$Q$ generated by~$\lbrack a,b\rbrack\coloneq a^{-1}b^{-1}ab$, for~$a$ in $A$ and~$b$ in~$B$. We consider~$A^p$ the normal subgroup of~$A$ generated by all~$p$-th powers in~$A$. Analogously to~\eqref{def filt}, we introduce the Zassenhaus filtration~$Q_\bullet(\F_p)$ and the lower central series~$Q_\bullet(\Z)$ of~$Q$. Let~$\mathbb{B}$ be either the ring~$\F_p$ or~$\Z$, both endowed with trivial filtration (or gradation). As in the~pro-$p$ case, we define~$\Ll(\mathbb{B},Q)$ and~$a_n(\mathbb{B},Q)$. Since~$Q$ is finitely generated, the numbers~$a_n(\mathbb{B}, Q)$ are well-defined for all positive integers~$n$. From a family~$G\coloneq \{ {}_1G,\dots ,{}_k G\}$ of finitely generated abstract groups, we obtain families~$\Ll(\mathbb{B},G)\coloneq \{\Ll(\mathbb{B},{}_1G),\dots, \Ll(\mathbb{B},{}_kG)\}$ of locally finite graded~$\mathbb{B}$-Lie algebras.  As in the pro-$p$ case, we define graph products of abstract groups and Lie algebras over~$\Z$.

We fix a prime~$p$, and we denote by~$\widehat{Q}^p$ the pro-$p$ completion of~$Q$. Observe that~$\widehat{Q}^p$ is a finitely generated pro-$p$ group. We define~$c_Q\colon Q\to \widehat{Q}^p$ the natural morphism from~$Q$ to its pro-$p$ completion, which induces two maps:
$$c_Q(\F_p)\colon \Ll(\F_p,Q)\to \Ll(\F_p, \widehat{Q}^p), \quad \text{and} \quad c_Q(\Z_p)\colon \Ll(\Z, Q)\otimes_{\Z}\Z_p\to \Ll(\Z_p, \widehat{Q}^p).$$ 
A well-known topological argument (following for instance from~\cite[Proposition~$1.19$]{DDMS}) shows that~$c_Q(\F_p)\colon \Ll(\F_p,Q)\to \Ll(\F_p, \widehat{Q}^p)$ is an isomorphism when~$Q$ is finitely generated. Let us also observe that a family~$G\coloneq \{ {}_1G,\dots ,{}_k G\}$ of finitely generated abstract groups, defines a family~$\widehat{G}^p \coloneq \{ \widehat{{}_1G}^p,\dots ,\widehat{{}_k G}^p\}$ of finitely generated pro-$p$ groups, and canonical morphisms that we denote by~${}_ic \colon {}_iG \to \widehat{{}_iG}^p$ rather than~$c_{{}_iG}$. We say that~$G$ satisfies the condition~$(\ast_\Z)$, if for every $1\leq i \leq k$, and every prime~$p$, we have the following condition:
\begin{equation}\tag{$\ast_\Z$}
\left\{ \begin{aligned}
(i) & \text{ the } \Z_p\text{-module } \Ll(\Z_p, \widehat{{}_iG}^p) \text{ is free,}
\\ (ii) & \text{ the map } {}_ic(\Z_p)\colon \Ll(\Z,{}_iG)\otimes_\Z \Z_p \to \Ll(\Z_p, \widehat{{}_iG}^p) \text{ is an isomorphism.}.
\end{aligned} \right.
\end{equation} 
We show in this paper the following result:

\begin{coco}[Theorem~\ref{abstract computations}]\label{application abstract}
Let $G\coloneq \{ {}_1G,\dots ,{}_k G\}$ be a family of finitely generated abstract groups and let~$p$ be a prime. 
%If for every~$1\leq i \leq k$ the group~${}_iG$ is residually~$p$-finite, then $\Gamma G$ is residually~$p$-finite. 
Then~$\Ll(\F_p, \Gamma G)\simeq \Gamma \Ll(\F_p,G)$.

Additionally, assume that $G$ satisfies~$(\ast_\Z)$, %then~$\Gamma G$ is residually~$p$-finite and
then~$\Ll(\Z, \Gamma G)$ is torsion-free and
$$\Ll(\Z,\Gamma G)\simeq \Gamma \Ll(\Z,G), \quad \text{and} \quad\Ll(\Z,\Gamma G)\otimes_\Z \Z_p\simeq \Ll(\Z_p, \Gamma \widehat{G}^p)  \simeq \Gamma \Ll(\Z_p, \widehat{G}^p).$$ 
Thus, for every positive integer $n$, and every prime~$p$, we have~$a_n(\Z,\Gamma G)=a_n(\Z_p,\Gamma \widehat{G}^p)$.
\end{coco}

%Green~\cite[Theorem~$5.6$]{ruthgraph} showed that if for a fixed prime~$p$ and~$1\leq i \leq k$, the group~${}_iG$ is residually~$p$-finite, then~$\Gamma G$ is residually~$p$-finite. This fact was also known (see~\cite[Corollary~$2.10$]{ruthgraph}) for Right Angled Artin groups (RAAGs). This means a graph product of a family~$G$ given by~${}_1G = {}_2G = \dots= {}_kG\coloneq \Z$. 

The isomorphism $\Ll(\F_p, \Gamma G)\simeq \Gamma \Ll(\F_p,G)$ was previously proved in the following cases:
\\$(i)$ The graph $\Gamma$ is free or complete. We refer to~\cite{Lichtman} and~\cite{leoni2024zassenhaus} for generalizations.
\\$(ii)$ The group~$\Gamma G$ is a RAAG. For instance, see~\cite[Theorem $6.4$]{wade2016lower}.
\\$(iii)$ The group~$\Gamma G$ is a graph product of a family~$G$ given by~${}_1G = {}_2G = \dots= {}_kG\coloneq \Z/p\Z$. This case was shown by Panov and Rahmatullaev~\cite{panovpseries}. When~$p=2$, we infer Coxeter groups.

The isomorphism $\Ll(\Z,\Gamma G)\simeq \Gamma \Ll(\Z,G)$ still holds for the case $(ii)$ of RAAGs. We refer to~\cite[Theorems $6.3$ and $6.4$]{wade2016lower}. The previous isomorphism also holds in the following case given by~\cite[Corollary]{Smel}. Assume that~$\Gamma^f$ is a free graph. We take~$G$ a family of groups such that, for every~$1\leq i \leq k$, $(a)$ the $\Z$-module~$\Ll(\Z,{}_iG)$ is torsion-free and $(b)$~we have the equality~$\bigcap_{n\in \N}{}_iG_n(\Z)=\{1\}$. It was shown that~$\Ll(\Z, \Gamma^f G)\simeq \Gamma^f \Ll(\Z,G)$. The condition~$(a)$ is indeed necessary. Let~$\Gamma G$ be a Coxeter group (as described in the case~$(iii)$). Then the algebra~$\Ll({\Z},\Gamma G)$ has torsion. See~\cite[Proposition $4.1$]{veryovkincoxeter}. Veryovkin also showed \cite[Example $4.3$]{veryovkincoxeter}, for $\Gamma^{f,2}$ the free graph on two vertices, that 
$$\Ll_3({\Z},\Gamma^{f,2} G)\not\simeq (\Gamma^{f,2} \Ll)_3({\Z},G)\coloneq (\Gamma^{f,2}G)_2(\Z)/(\Gamma^{f,2}G)_3(\Z),$$

Using geometric tools, the cohomology groups over~$\F_p$ (with the trivial action) of graph products of abstract groups were computed in~\cite[Theorem~2.35]{bahri2010polyhedral}. In this paper, we relate the cohomology of graph products in the abstract and pro-$p$ cases. We say that a finitely generated group~$Q$ is~$p$-cohomologically complete if for every positive integer~$n$ we have~$H^n(Q)\simeq H^n(\widehat{Q}^p)$.
%, and~$(b)$ the group~$Q$ is residually~$p$-finite. 
Labute~\cite{Labute1967} %(and~\cite[§1.3]{Baumslag}) 
gave first examples of~$p$-cohomologically complete groups. This notion was later studied by~Lorensen \cite{LORENSEN20106}. He showed that RAAGs are~$p$-cohomologically complete. He also proved that this property is stable by some HNN extensions and some amalgamated products. %We refer to~\cite[§9.2]{RibesZal} for further details on amalgamated products. 
In this paper, we show:

\begin{prp}[Corollary~\ref{cohomological completness} and Remark~\ref{computation cohomology abstract groups}]\label{main complete coho}
Let~$G\coloneq \{{}_1G, \dots, {}_kG\}$ be a family of finitely generated abstract groups. Let~$p$ be a prime. Then, for every positive integer~$n$:
$$h^n(\Gamma G)=\sum_{m\geq 1}\sum_{ (i_1,\dots, i_m) \in \Gamma_m}\sum_{n_1+\dots+n_m=n} h^{n_1}({}_{i_1}G)\dots h^{n_m}({}_{i_m}G).$$

Furthermore, assume that for every~$1\leq i \leq k$ the group~${}_iG$ is $p$-cohomologically complete. Then the group~$\Gamma G$ is $p$-cohomologically complete.
\end{prp}

\subsection*{Example}
Let~$\{{}_1g,\dots,{}_kg\}$ be a collection of positive integers. We denote by $S_{{}_ig}$ a surface of genus~${}_ig$. We define~${}_iG$ the fundamental group of~$S_{{}_ig}$. Let us write~$S\coloneq\{S_{{}_1g},\dots, S_{{}_kg}\}$ and~$G\coloneq \{{}_1G,\dots, {}_kG\}$. Let~$\Gamma$ be an undirected graph. From~\cite[Proposition~$2.9$]{panovpseries}, there exists a pointed topological space, that we denote by~$\Gamma S$, with fundamental group~$\Gamma G$. The space~$\Gamma S$ is constructed as the graph product of the family~$S$ by~$\Gamma$. We refer to~\cite[Construction~$2.1$]{panovpseries}. In this example, we study~$\Gamma G$. Let us first observe from~\cite[Theorem~$5.6$]{ruthgraph} and~\cite[§1.3]{Baumslag} that the group~$\Gamma G$ is residually~$p$-finite, i.e.\ the group~$\Gamma G$ injects into its pro-$p$ completion~$\widehat{\Gamma G}^p$. Equivalently, we infer~$\bigcap_{n\in \NN}(\Gamma G)_n(\F_p)=\{1\}$.

%The family~$G$ satisfies the condition~$(iii)$ from Corollary~\ref{application abstract} (see \cite[§1.3]{Baumslag}). 
%Thus, for every prime~$p$, we have a group injection~$S_{{}_ig}\hookrightarrow \widehat{S_{{}_ig}}^p$. 
For every~$1\leq i \leq k$ and every prime~$p$, both the abstract group~${}_iG$, and the pro-$p$ group~$\widehat{{}_iG}^p$, admit presentations:
\begin{equation}\tag{$\P_{{}_ig}$} 
\langle {}_ix_1,\dots ,{}_ix_{{}_ig}, {}_iy_1,\dots, {}_iy_{{}_ig}|\quad \prod_{j=1}^{{}_ig}[{}_ix_j,{}_iy_j]=1\rangle.
\end{equation}
Since~$\widehat{{}_iG}^p$ is a Demushkin group, it follows from~\cite[Theorem~C]{minac2021koszul}, that the algebra~$\E(\F_p, \widehat{{{}_i}G}^p)$ is Koszul. Moreover
$$H^\bullet(\widehat{{}_iG}^p, t)=1+2{}_ig t+t^2, \quad gocha(\F_p,\widehat{{}_iG}^p,t)=\frac{1}{1-2{}_igt+t^2}.$$
By Labute \cite{labutedescending}, the family $G$ also satisfies the condition $(i)$ and $(ii)$ of Corollary~\ref{application abstract}. 
%In particular, for every prime~$p$ and every~$1\leq i \leq k$, we have:
%$$\Ll(\Z,S_{{}_ig})\otimes_\Z \Z_p \simeq \Ll(\Z_p, \widehat{S_{{}_ig}}^p).$$
Consequently, for every graph~$\Gamma$, the algebra $\Ll(\Z,\Gamma G)$ is torsion-free. Thus Proposition~\ref{computekosz} provides
$$\Ll(\Z,\Gamma G)\simeq \Gamma \Ll(\Z,G), \quad \text{and } gocha(\Gamma \widehat{G}^p,t)= \frac{1}{H^\bullet(\Gamma \widehat{G}^p,-t)}, \text{ for every prime } p.$$

Furthermore, we deduce from Labute~\cite{Labute1967} that for every~$1\leq i \leq k$, the group~${}_iG$ is~$p$-cohomologically complete. As a consequence of Proposition~\ref{main complete coho}, the group~$\Gamma G$ is~$p$-cohomologically complete. Hence, we show that, for every positive integer~$n$:
$$h^n(\Gamma G)=h^n(\Gamma \widehat{G}^p)=\sum_{m\geq 1}\sum_{( i_1,\dots, i_m)\in \Gamma_m}\sum_{n_1+\dots+n_m=n} h^{n_1}({{}_{i_1}G})\dots h^{n_m}({{}_{i_m}G}),$$
where $h^0({{}_iG})=1$, $h^1({{}_iG})=2{}_ig$, $h^2({{}_iG})=1$ and $h^n({{}_iG})=0$ for~$n\geq 3$.

Finally, let us also observe that~Corollary~\ref{application abstract} yields for every positive integer~$n$ and every prime~$p$ the equality~$a_n(\Z_p,\Gamma \widehat{G}^p)=a_n(\Z,\Gamma G)$.
\begin{Exa}
As a concrete example, let us consider the graph $\Gamma$ given by:

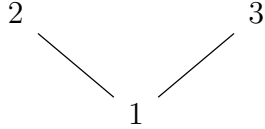
\centering{% https://tikzcd.yichuanshen.de/#N4Igdg9gJgpgziAXAbVABwnAlgFyxMJZARgBpiBdUkANwEMAbAVxiRGJAF9T1Nd9CKAAykhVWoxZsATFx4gM2PASLTR4+s1aIQAZi7iYUAObwioAGYAnCAFskIkDghIyIABYw6UNpDCtuSxt7REdnJDUPLx8dPwCKTiA
\begin{tikzcd}
2 &                                           & 3 \\
  & 1 \arrow[lu, no head] \arrow[ru, no head] &  
\end{tikzcd}}

\justifying
We take ${}_1g={}_2g={}_3g=2$. For~$1\leq i \leq k$, the group~${{}_iG}$ is presented~by:
\begin{equation}\tag{$\P_{{{}_iG}}$} 
\langle {}_ix_1 ,{}_ix_{2}, {}_iy_1, {}_iy_{2}|\quad {}_il\coloneq [{}_ix_1,{}_iy_1][{}_ix_2,{}_iy_2]=1\rangle.
\end{equation}

The group $\Gamma G$ (resp.\ $\Gamma \widehat{G}^p$) is an amalgamated product and admits a presentation $(\P_{\Gamma G})$ (resp.\ a pro-$p$ presentation~$(\P_{\Gamma \widehat{G}^p})$) given by 

$\bullet$ $12$ generators $\{ {}_ix_j, {}_iy_j\}$ for $1\leq i\leq 3$ and $1\leq j\leq 2$,

$\bullet$ $35$ relations: ${}_1l$, ${}_2l$, ${}_3l$, $[{}_ux_a, {}_vx_b]$, $[{}_uy_a, {}_vy_b]$ and $[{}_ux_a, {}_vy_b]$ for $\{u,v\}=\{1,2\}$ or $\{u,v\}=\{1,3\}$ and $1\leq a,b \leq 2$.
%\\For every prime~$p$, the group~$\Gamma G$ is~$p$-cohomologically complete. 
\\The dimensions of the cohomology groups of~$\Gamma G$ over~$\F_p$ (with trivial action) are given by:
%(of $\Gamma\widehat{G}^p$ for every $p$, over $\F_p$) are given by:
%$$h^1(\Gamma \widehat{G}^p)=12, \quad h^2(\Gamma \widehat{G}^p)=35, \quad h^3(\Gamma \widehat{G}^p)=16, \quad h^4(\Gamma \widehat{G}^p)=2, \text{ and } h^n(\Gamma \widehat{G}^p)=0 \text{ for } n\geq 5.$$
%Since~$\Gamma G$ is $p$-cohomologically complete, we obtain:
$$h^1(\Gamma G)=12, \quad h^2(\Gamma G)=35, \quad h^3(\Gamma G)=16, \quad h^4(\Gamma G)=2, \text{ and } h^n(\Gamma G)=0 \text{ for } n\geq 5.$$

%As a consequence, for every prime $p$, the presentation $(\P_{\Gamma \widehat{G}^p})$ is minimal. 
For every prime~$p$, the gocha and Poincaré series of $\Gamma \widehat{G}^p$ are given by:
$$gocha(\Gamma \widehat{G}^p,t)=\frac{1}{1-12t+35t^2-16t^3+2t^4}, \quad H^\bullet(\Gamma \widehat{G}^p, t)=1+12t+35t^2+16t^3+2t^4.$$

From Theorem~\ref{answerminac}, we obtain for instance:
\begin{multline*}
a_1(\Z,\Gamma G)=12, \quad a_2(\Z,\Gamma G)=31, \quad a_3(\Z,\Gamma G)=168, \quad a_4(\Z,\Gamma G)=928, 
\\ \text{and } \quad a_5(\Z,\Gamma G)=5704.
\end{multline*}
\end{Exa}

\subsection*{Structure of the paper}
We begin with some prerequisites on graded and filtered algebras, as well as the Magnus isomorphism. Then the first section contains the proof of Theorem~\ref{answerminac}. The second section provides various new examples fitting the conditions of Theorem~\ref{answerminac}. This section establishes Theorem~\ref{main result}. It is based on arguments from~\cite[Chapter~$4$]{leoni2024zassenhaus}. The third section studies the cohomology of graph products of pro-$p$ groups. The main argument follows from Koszulity and this section shows Proposition~\ref{computekosz}. Finally, the last section generalizes the results on graph products of pro-$p$ groups to the case of finitely presented abstract groups. It yields proofs of Corollary~\ref{application abstract} and Proposition~\ref{main complete coho}.

\tableofcontents

As this paper is technical in nature, we provide a table of notations to assist the reader. We restrict our attention to pro-$p$ groups. The corresponding notations in the abstract group setting are given in the last section. For a fixed prime~$p$, all of the graded~$\F_p$-Lie algebras are assumed to be~$p$-restricted.

\newpage
\begin{table}[H]
\raggedright
\renewcommand{\arraystretch}{1.2}
\begin{tabular}{@{}ll@{}}
\toprule
\textbf{Symbol} & \textbf{Meaning} \\
\midrule
$p$ & A prime number. \\
$\AA$ & Either $\mathbb{F}_p$ (field with $p$ elements) or $\mathbb{Z}_p$ (ring of $p$-adic integers). \\
$[x,y]$ & Commutator $x^{-1}y^{-1}xy$. \\
%$\operatorname{rank}_\AA V$ & Rank of a free finitely generated $\AA$-module $V$ \\
$E(\AA)$ & Noncommutative series over~$\{X_1,\dots, X_d\}$, filtered by~$\deg(X_i)=1$ for~$1\leq i \leq d$. \\
$E_n(\AA)$ & The $n$-th power of the augmentation ideal. \\
$\E(\AA)$ & Noncommutative polynomials over~$\{X_1,\dots, X_d\}$, graded by~$\deg(X_i)=1$. \\
$\Ll(\AA)$ & Free graded $\AA$-Lie algebra on~$\{X_1,\dots, X_d\}$, with~$\deg(X_i)=1$. \\
$\mathcal{U}(\Ll)$ & Universal enveloping algebra of a graded $\AA$-Lie algebra $\Ll$. \\
%$\omega_\A(a)$ & Degree of $a$ in a filtered/graded algebra $\A$ \\
%$\grad$ & Functor from filtered $\AA$(-Lie)-algebras to graded $\AA$(-Lie)-algebras \\
$F$ & Free pro-$p$ group with~$d$ generators. \\
$R$ & Normal closed subgroup of~$F$ generated by~$r$ relations $\{l_1,\dots, l_r\}$. \\
$Q \coloneq F/R$ & Finitely presented pro-$p$ group with~$d$ generators and~$r$ relations. \\
%$[A,B]$ & Closed subgroup generated by commutators %$[a,b]$ with $a \in A$, $b \in B$ \\
%$AB$ & Closed subgroup generated by products $ab$ with $a \in A$, $b \in B$ \\
%$A^p$ & Closed subgroup generated by $a^p$ with $a \in A$ \\
$\psi_\AA$ & Magnus isomorphism of $\AA$-filtered algebras between
\\ & the completed group algebra of~$F$ over~$\AA$ and~$E(\AA)$. \\
$I(\AA,R)$ & Closed two-sided ideal of~$E(\AA)$ generated by~$\psi_{\AA}(r-1)$ with~$r\in R$.\\
$E(\AA,Q)$ & Quotient of~$E(\AA)$ by~$I(\AA,R)$. \\
$E_n(\AA,Q)$ & The $n$-th power of the augmentation ideal of $E(\AA,Q)$. \\
$\psi_{\AA,Q}$ & Isomorphism of $\AA$-filtered algebras (from Magnus) between
\\ & the completed group algebra of~$Q$ over~$\AA$ and~$E(\AA,Q)$. \\
$Q_n(\mathbb{Z}_p)$ & The $n$-th term of the lower central series of $Q$. \\
$Q_n(\mathbb{F}_p)$ & The $n$-th term of the Zassenhaus filtration of $Q$. \\
$\E(\AA,Q)$ & Quotient of~$\E(\AA)$ isomorphic to $ \bigoplus_{n\in \NN} E_n(\AA,Q)/E_{n+1}(\AA,Q)$. \\
$c_n(\AA,Q)$ & $\AA$-Rank of $E_n(\AA,Q)/E_{n+1}(\AA,Q)$ \\
$gocha(Q,t)$ & Gocha (Golod-Shafarevich) series of~$Q$ given by~$\sum_n c_n(\F_p,Q) t^n$. \\ 
% & with~$c_n(\AA, Q)\coloneq \rk_{\AA}E_n(\AA,Q)/E_{n+1}(\AA,Q)$. \\
$\Ll(\AA,Q)$ & Quotient of~$\Ll(\AA)$ isomorphic to~$\bigoplus_{n\in \NN} Q_n(\AA)/Q_{n+1}(\AA)$. \\
%$\Ll_n(\AA,Q)$ & $\AA$-module of finite rank isomorphic to~$Q_n(\AA)/Q_{n+1}(\AA)$ \\
$a_n(\AA,Q)$ & Rank of $\Ll_n(\AA,Q)$ over $\AA$. \\
$H^\bullet(Q)$ & Continuous group cohomology of $Q$ over $\mathbb{F}_p$. \\
%$h^n(Q)$ & Dimension $\dim_{\mathbb{F}_p}(H^n(Q))$ \\
$H^\bullet(Q,t)$ & Poincaré series of~$Q$, given by $\sum_n h^n(Q)t^n$, with $h^n(Q)\coloneq \dim_{\mathbb{F}_p}(H^n(Q))$. \\
$\Gamma=(\bX,\bE)$ & Undirected graph with set of vertices $\bX=\{1,\dots ,k\}$ and set of edges $\bE$. \\
$\Gamma_n$ & Set of $n$-cliques (complete subgraphs with $n$ vertices) of $\Gamma$. \\
$(i_1,\dots, i_n)\in \Gamma_n$ & An $n$-cliques of $\Gamma$, satisfying~$i_1<i_2<\dots <i_n$. \\
${}_iG$ & A finitely presented pro-$p$ group.\\
$\{{}_ix_j\}_{j=1}^{{}_id}$, $\{{}_il_j\}_{j=1}^{{}_ir}$ & Systems of generators and relations of~${}_iG$.\\
$G \coloneq \{{}_iG \}_{i=1}^k$ & A family of finitely presented pro-$p$ groups. \\
$\E(\AA,G)$ & Family $\{\E(\AA,{}_iG) \}_{i=1}^k$ of graded locally finite $\AA$-algebras. \\
$\Ll(\AA,G)$ & Family $\{\Ll(\AA,{}_iG) \}_{i=1}^k$ of graded locally finite Lie-$\AA$-algebras. \\
$\A\coloneq \{{}_i\A\}_{i=1}^k$ & $\A$ is either $G$, $\E(\AA, G)$ or $\Ll(\AA,G)$. \\
$\Gamma \A$ & Graph product of $\A$ by $\Gamma$. \\
${}'\Gamma\coloneq ({}'\bX, {}'\bE)$, ${}'\A$ & ${}'\Gamma$ is a subgraph of~$\Gamma$ and~${}'\A$ is the subfamily of~$\A$ indexed by~${}'\bX$. \\
$\iota_{{}'\Gamma/\Gamma}$ & Canonical map from~${}'\Gamma {}'\A\to \Gamma \A$. \\
%$\iota_i$ & Canonical injection ${}_i\A \to \Gamma \A$ in a graph product \\
\bottomrule
\end{tabular}
\caption{Summary of notations used in the paper.}
\end{table}

\newpage

\section*{Prerequisites on algebras and the Magnus isomorphism}
We introduce some well-known facts on graded/filtered $\AA$-Lie algebras. For further background, we suggest that the reader consult~\cite{lazard1965groupes}, \cite{HAMZA2023172}, \cite{leoni2024zassenhaus} and~\cite{Hamza25}. We only study pro-$p$ groups in this section. We refer to the last section for results in the abstract case.

\subsection*{Pro-$p$ groups and Lie algebras} In this subsection, we recall some facts on pro-$p$ groups and some related Lie algebras.

%If $Q$ is a pro-$p$ group, we denote by $[A;B]$ the closure of the commutator subgroup of $A$ and $B$ previously introduced.

%which maps a $\AA$-filtered algebra $A, \{A_n\}_{n\in \NN}$ to $\grad(A)\coloneq \bigoplus_{n\in \NN} A_n/A_{n+1}$. For every filtered map~$f\colon A, \{A_n\}_{n\in \NN} \to B, \{B_n\}_{n\in \NN}$, we define $\grad(f)\coloneq \grad(A)\to \grad(B)$.

$\bullet$ The ring~$\AA$ is either seen as a filtered ring (with trivial filtration) or as a graded ring (with trivial gradation). This allows us to make the abuse of notations~$\grad(\AA)=\AA$.
A filtered algebra (resp.\ graded, resp.\ graded Lie) is said to be over~$\AA$ if it is an $\AA$-module. If $(A, \{A_n\}_{n\in \NN})$ is a filtered algebra, we define the degree of an element~$a\in A$, denoted by $\omega_A(a)$, as the least integer $n$ such that $a$ is in $A_n\setminus A_{n+1}$. We also have a notion of degree of graded rings. We say that~$\phi\colon A \to B$ is a morphism of filtered algebra if $\omega_B(\phi(a))\geq \omega_A(a)$. A map $\phi$ is strictly filtered if for every~$b\in {\rm Im}(\phi)$ there exists $a\in A$ such that $\phi(a)=b$ and $\omega_A(a)=\omega_B(b)$.

$\bullet$ We assume the extra condition that~$\F_p$-graded Lie algebras are restricted, i.e.\ endowed with an operation~${\bullet}^p$ satisfying some specific properties (see~\cite[Section~$12.1$]{DDMS} for further references). A graded (Lie) $\AA$-algebra~$\A\coloneq \bigoplus_{n\in \NN} \A_n$ is called locally finite if~$\rk_\AA \A_n$ is finite for every positive integer~$n$. In this paper all of our graded (Lie)~$\AA$-algebras are locally finite.

$\bullet$ We have a functor~$\grad$ from the category of $\AA$-filtered algebras to $\AA$-graded algebras. Let us refer to \cite[Chapitre Premier, §$1,2$]{lazard1965groupes} for further details. We denote by~$\U(\A)$ the $\AA$-universal envelope of a $\AA$-graded Lie algebra $\A$.

$\bullet$ Let $Q$ be a finitely presented pro-$p$ group with set of generators $\{x_1,\dots,x_d\}$ and relations $\{l_1,\dots,l_r\}$.
We fix a presentation:
\begin{equation}\tag{$\P_Q$}
1\to R \to F \to^{\pi_Q} Q \to 1.
\end{equation}

$\bullet$ Assume that each variable $X_u$ of~$E(\AA)$ is endowed with degree~$1$. Then~$E(\AA)$ is a filtered algebra, and we denote its filtration by~$\{E_n(\AA)\}_{n\in \NN}$.
Magnus~\cite[Chapitre~II, Partie~$3$]{lazard1965groupes} gave an isomorphism of $\AA$-filtered algebras~$\psi_\AA$ between the completed group algebra of~$F$ over~$\AA$ and~$E(\AA)$. For every~$1\leq i \leq d$, the morphism~$\psi_\AA$ sends~$x_i$ to~$1+X_i$. We denote by~$I(\AA,R)$ the closed two sided ideal of $E(\AA)$ generated by $\psi_{\AA}(r-1)$ for $r$ in~$R$. We define~$\I(\AA,R)\coloneq \bigoplus_{n\in \NN} \I_n(\AA,R))$, where $I_n(\AA,R)\coloneq I(\AA,R)\cap E_n(\AA)$. Consider $E(\AA,Q)\coloneq E(\AA)/I(\AA,R)$ and~$\E(\AA,Q)\coloneq \E(\AA)/\I(\AA,R)$. The $\AA$-algebra~$E(\AA,Q)$ is endowed with a (quotient) filtration, that we denote by~$\{E_n(\AA,Q)\}_{n\in \NN}$. We also have~$\E(\AA,Q)\coloneq \bigoplus_{n\in \NN}\E_n(\AA,Q)$ where~$\E_n(\AA,Q)\coloneq E_n(\AA,Q)/E_{n+1}(\AA,Q)$. From the Magnus isomorphism and the presentation~$(\P_Q)$, we infer an isomorphism of filtered~$\AA$-algebras~$\psi_{\AA,Q}$ between~$E(\AA,Q)$ and the completed group algebra of~$Q$ over~$\AA$. We deduce presentations
\begin{equation*}
\begin{aligned}
0 \to I(\AA, R)\to^{\iota_{\AA,Q}} E(\AA)\to^{e\pi_{\AA,Q}} E(\AA,Q)\to 0,
\\ 0\to \I(\AA, R) \to^{\grad(\iota_{\AA,Q})} \E(\AA)\to^{\grad(e\pi_{\AA,Q})} \E(\AA,Q)\to 0.
\end{aligned}
\end{equation*}
  Observe that~$(E(\AA,Q),\{E_n(\AA,Q)\}_{n\in \NN})$ is isomorphic to the completed group algebra of~$Q$ over $\AA$ filtered by the $n$-th power of the augmentation ideal. 
  %and $\E(\AA,Q)\coloneq \bigoplus_n \E_n(\AA,Q)$, where $\E_n(\AA,Q)\coloneq E_n(\AA,Q)/E_{n+1}(\AA,Q)$. 

$\bullet$ From Magnus isomorphism, we deduce an isomorphism of graded Lie $\AA$-algebras:~$\Ll(\AA) \simeq \bigoplus_n F_n(\AA)/F_{n+1}(\AA)$. We define $\Ll(\AA, Q)$ from~$(\P_Q)$ by the quotient of~$\Ll(\AA)$ isomorphic, through the Magnus isomorphism, to~$\bigoplus_{n\in \NN} \Ll_n(\AA,Q)$ where $\Ll_n(\AA,Q)\coloneq Q_n(\AA)/Q_{n+1}(\AA)$. In this paper, we assume (or show) that $\Ll(\AA,Q)$ is free over~$\AA$.  Since~$Q$ is finitely generated, we define
$$c_n(\AA,Q)\coloneq \rk_\AA \E_n(\AA,Q),\quad \text{and} \quad a_n(\AA,Q)\coloneq \rk_\AA \Ll_n(\AA,Q).$$
From~\cite[Theorem~$1.3$]{hartlfox} and~\cite[Corollary~$1.4$]{HAMZA2023172}, if $\Ll(\AA,Q)$ is torsion-free over $\AA$, then~$\U(\Ll(\AA,Q))\simeq \E(\AA,Q)$, which is also torsion-free by the PBW Theorem.

$\bullet$ When $\AA$ is clear from the context, in order to simplify the notations, we prefer to write $E(Q)$, $\E(Q)$, $\Ll(Q)$, $I(R)$, $\I(R)$... rather than $E(\AA, Q)$, $\E(\AA, Q)$, $\Ll(\AA, Q)$, $I(\AA, R)$, $\I(\AA, R)$...  

\subsection*{Graph theory and algebraic structures} 
%Let $\Gamma$ be an undirected graph with set of vertices $\bX\coloneq \{1;\dots; k\}$ and set of edges $\bE$. 
We denote an~$n$-clique of $\Gamma$, i.e.\ a maximal complete subgraph of~$\Gamma$ with~$n$ vertices, by its $n$-uplet vertices $(i_1,\dots, i_n)$. We order this~$n$-uplet of vertices by~$i_j<i_{j+1}$. We define~$\Gamma_n$ the set of~$n$-cliques of~$\Gamma$. 
\\We consider $\A\coloneq \{ {}_1\A,\dots,{}_k\A\}$ a family of objects in one of the following categories:
\begin{enumerate}
%  \item finitely generated abstract groups, where $[\bullet , \bullet]$ denotes the commutator, and $\langle \bullet \rangle$ is the normal subgroup generated by $\bullet$,
\item finitely generated pro-$p$ groups, where $[\bullet , \bullet]$ denotes the commutator, and $\langle \bullet \rangle$ is the normal closed subgroup generated by $\bullet$,
\item (locally finite) $\AA$-graded Lie algebras, where $[\bullet , \bullet]$ denotes the Lie-bracket, and $\langle \bullet \rangle$ is the $\AA$-Lie-ideal generated by $\bullet$,
\item finitely generated $\AA$-filtered algebras, where $[X , Y]\coloneq XY-YX$ for $X$ and $Y$ elements in a finitely generated $\AA$-filtered algebra, and $\langle \bullet \rangle$ is the closed two-sided ideal generated by $\bullet$.
\end{enumerate} 
All of these categories are endowed with a coproduct. For every~$i\in \bX$, we have canonical maps~${}_i\A \to \coprod_{j=1}^k {}_j\A$. We define the graph product~$\Gamma \A$ of the family $\A$ by $\Gamma$~by:
$$\Gamma\A\coloneq \coprod_{i=1}^k {}_i\A/\langle \lbrack \iota_u({}_ua),\iota_v({}_va)\rbrack| \quad \{u,v\}\in \bE \rangle.$$
Let~${}'\Gamma\coloneq ('\bX, '\bE)$ be a subgraph of~$\Gamma$. We define the subfamily~${}'\A\coloneq \{{}_j\A\}_{j\in {}'\bX}$ of~$\A$ indexed by~${}'\bX$. This allows us to define the graph product~${}'\Gamma ('\A)$. We also infer canonical maps~$\iota_{{}' \Gamma/ \Gamma}\colon {}'\Gamma ({}'\A)\to \Gamma \A$, and~$\iota_{\Gamma/ {}'\Gamma}\colon \Gamma \A \to {}'\Gamma {}'\A$.

Let $G\coloneq \{ {}_1G,\dots ,{}_k G\}$ be a family of finitely presented pro-$p$ groups. For~$1\leq i \leq k$, we denote by $\{{}_ix_1,\dots, {}_ix_{{}_id}\}$ a set of generators and~$\{ {}_il_1,\dots, {}_il_{{}_ir}\}$ a set of relations of~${}_iG$, for~$1\leq i \leq k$. We have a free presentation~${}_iG\simeq {}_iF/{}_iR$. Let~${}_i\E(\AA)$ be the graded algebra on~${}_iX_j$ over~$\AA$ where~$1\leq j \leq {}_id$, and every variable~${}_iX_j$ has degree one. We also define~$\I(\AA, {}_iR)$ the two sided ideal of~${}_i\E(\AA)$ generated, through the Magnus isomorphism, by~$\grad({}_iR)$. This allows us to define~$\E(\AA,{}_iG)\coloneq {}_i\E(\AA)/\I(\AA,{}_iR)$. 

\section{Proof of Theorem \ref{answerminac} and some applications}
%Furthermore (see for instance \cite[Corollary $1.4$]{HAMZA2023172}) the algebra $\E(\Z_p,Q)$ is the universal envelope of $\Ll(\Z_p,Q)$ and is also torsion-free by the PBW Theorem.
%$(i)$ $\Ll(\gamma,Q)\simeq \Ll(\Z_p,Q)$ and $(ii)$ the $\Z_p$-graded Lie algebra $\Ll(\Z_p,Q)$ is torsion free. 
This section proves Theorem~\ref{answerminac} using techniques from the author~\cite{HAMZA2023172}, Lazard~\cite{lazard1965groupes} and Leoni~\cite[Chapter $1$, §1.1-§1.3]{leoni2024zassenhaus}. Then we apply these techniques to infer Proposition~\ref{ast and graph}, which is a result on the torsion-freeness of graph products over some algebras.

We assume here that $Q$ is a finitely generated pro-$p$ group and $\Ll(\Z_p,Q)$ is torsion-free. Observe that \cite[Corollary $4.2$]{QUILLEN1968411} allows us to recover the notations from~\cite{HAMZA2023172}, i.e.\ for every positive integer~$n$: 
$$Q_n(\Z_p)=\{q\in Q\mid  \psi_{\Z_p,Q}(q-1)\in E_n(\Z_p,Q) \}.$$

%Finally, form a family $G$, we naturally infer families $\Ll(\AA,G)$, $\E(\AA,G)$ and $E(\AA,G)$ which allows us to define $\Gamma \Ll(\AA,G)$, $\Gamma \E(\AA,G)$ and $\Gamma E(\AA,G)$.

\subsection{Proof of Theorem~\ref{answerminac}}
Strictly filtered maps play a fundamental role in our proof. We refer to the previous section and~\cite[Chapter $1$, §1.1-§1.3]{leoni2024zassenhaus} for general references on them. Recall, from Equations \eqref{equation minac} and \eqref{equation minac2}, that it is sufficient to show the following result. 

\begin{theo}\label{answerminac2}
For every positive integer~$n$, we have~$c_n(\F_p,Q)=c_n(\Z_p,Q)$. 
\end{theo}

\begin{rema}\label{Marmo-Riley-Weigel}
Let us also note that Theorem~\ref{answerminac2} can be deduced from \cite[Theorem~B]{WeigelRiley} and the associated universal enveloping algebras. Here we present an alternative proof that does not involve~$p$-restricted Lie algebras.
\end{rema}
%The main references for this proof are the work of: the author~\cite{HAMZA2023172}, Lazard~\cite{lazard1965groupes} and Leoni~\cite[Chapter $1$, §1.1-§1.3]{leoni2024zassenhaus}.
%We define $E^0(\AA,Q):=\AA\lbrack Q\rbrack$ the group algebra over $\AA$ of $Q$, and we endow it with filtration $E_n^0(\AA,Q)$ the $n$-th power of the kernel of the augmentation map defined by:
%$$E^0(\AA,Q)\to \AA; \quad \sum_q a_q q \mapsto \sum_q a_q.$$
%The completion of the space $E^0(\AA,Q)$ is $E(\AA,Q)$ and since $Q$ is finitely generated, we have $\grad(E^0(\AA,Q))\simeq \E(\AA,Q)$. Furthermore, we have an exact sequence:
%$$0\to I^0(\AA)\to E^0(\AA,F)\to E^0(\AA,Q)\to 0,$$
%where $I^0(\AA)$ is the two sided ideal of $E^0(\AA,F)$ generated by $r-1, r\in R$, and endowed with the induced filtration defined for every integer $n$ by $I^0_n(\AA)=I^0(\AA)\cap E_n^0(\AA,F)$.
Let us introduce the group algebra $\AA[Q]$, which is, through the Magnus isomorphism, dense in $E(\AA,Q)$. Thus, we have a map 
$$\phi_{p,Q}\colon E(\Z_p,Q)\to E(\F_p,Q); \quad \text{defined by }\sum_{q\in Q} a_q q \in \Z_p\lbrack Q\rbrack \mapsto \sum_{q\in Q} \overline{a_q}q\in \F_p\lbrack Q\rbrack,$$
 where $\overline{a_q}$ is the image of $a_q$ in $\Z_p$ modulo $p\Z_p$. This map is a surjection of filtered~$\Z_p$-algebras. 
 %and since $\E(\Z_p,Q)$ is torsion free, then if $g$ is in $E_n^0(\F_p,Q)\setminus E_{n+1}^0(\F_p,Q)$, there exists an element $h\in E_n^0(\Z_p,Q)\setminus E_{n+1}^0(\Z_p,Q)$ such that $g=\phi_p(h)$. Therefore the map
%$\grad(\phi_p)\colon \E^0(\Z_p,Q)\to \E^0(\F_p,Q)$ is surjective. 
By abuse of notation, we denote by $\phi_p$ the map $\phi_{p,F}\colon E(\Z_p)\to E(\F_p)$. This application maps a series $P\coloneq \sum_w a_w w$ to $\phi_p(P)\coloneq \sum_w \overline{a_w}w$, for $w$ monomials, $a_w$ coefficients in $\Z_p$ and $\overline{a_w}$ in $\F_p$ the image of $a_w$ modulo~$p\Z_p$. This map is surjective and strictly filtered.
%or the restriction $\phi_{p,F|I^0(\Z_p)}\colon I^0(\Z_p) \to I^0(\F_p)$. 

Let us also note that the map~$e\pi_{\AA, Q}$ is defined, through the Manus isomorphism, by the application which sends~$\sum_{f\in F} a_f f$ in~$\AA \lbrack F \rbrack$ to the element~$\sum_{q\in Q} \left(\sum_{f\mid \pi_Q(f)=q} a_f\right) q$ in~$\AA\lbrack Q\rbrack$.

\begin{lemm}\label{strictfiltr}
We have the following commutative diagram, where all maps are strictly filtered:

\centering{% https://tikzcd.yichuanshen.de/#N4Igdg9gJgpgziAXAbVABwnAlgFyxMJZABgBpiBdUkANwEMAbAVxiRGJAF9T1Nd9CKAEzkqtRizYBJAHrEAFAB1FALQD6aAJRceIDNjwEiAFlHV6zVohABROUtUbt3XgYFEAbGfGW2dhcrqaKQAis66+vxGKAAc3haS1hwuenyGgshxQmIJVuw6rlEZXtnmEnn+DgBiGqHhhekmpKU+ibb2yjVaBalu0cgiLbnSHYpd9b1FRGRD5WzJYjBQAObwRKAAZgBOEAC2SGQgOBBIAIwp23tIAJzUx0gxFzv7iKd3J4giR3RYDGwAFhAIABrHqXF5xI4fADsdx+f2sgJBYOeSC+90QAGZqP8YHQoGwcAB3CC4-EIJ5XRCwqFILwgMkE6zE0l4qAU3TgpDY2mIYyUl70jEAVgF3PedLKvmsyjQ-ywGhAOLZhJJjI5m1Rnwl1KlbVl8sVyvxqtZ5JRVLevMhwxlijlCrQSoZKuZarZFIonCAA
\begin{tikzcd}
0 \arrow[r] &  I(\Z_p) \arrow[r, "\iota_{\Z_p, Q}", hook] &  E(\Z_p) \arrow[r, "e\pi_{\Z_p,Q}", two heads] \arrow[dd, "\phi_p", two heads] &   {E(\Z_p,Q)} \arrow[r] \arrow[dd, "\phi_{p,Q}", two heads] &   0 \\
               &                                                              &                                                                   &                                                          &    \\
0 \arrow[r]   & I(\F_p) \arrow[r,"\iota_{\F_p, Q}", hook]                                 &   E(\F_p) \arrow[r, "e\pi_{\F_p,Q}", two heads]                                   & {E(\F_p,Q)} \arrow[r]                                 &   0
\end{tikzcd}}

\justifying
Thus we deduce the commutative diagram:

\centering{
\begin{tikzcd}
0 \arrow[r] &  \I(\Z_p) \arrow[r, "\grad(\iota_{\Z_p,Q})", hook] &  \E(\Z_p) \arrow[r, "\grad(e\pi_{\Z_p,Q})", two heads] \arrow[dd, "\grad(\phi_p)", two heads] &   {\E(\Z_p,Q)} \arrow[r] \arrow[dd, "\grad(\phi_{p,Q})", two heads] &   0 \\
               &                                                              &                                                                   &                                                          &    \\
0 \arrow[r]   & \I(\F_p) \arrow[r,"\grad(\iota_{\F_p, Q})", hook]                                 &   \E(\F_p) \arrow[r, "\grad(e\pi_{\F_p,Q})", two heads]                                   & {\E(\F_p,Q)} \arrow[r]                                 &   0
\end{tikzcd}}
\end{lemm}

\justifying
\begin{proof}
The first diagram is commutative (see for instance~\cite[Proposition~$4.3$]{Forre}). Furthermore, by definition, the maps defined in the rows are all strictly filtered.

Since the map $\phi_p\colon E(\Z_p)\to E(\F_p)$ is surjective and strictly filtered, then from~\cite[Lemma $1.1.5$]{leoni2024zassenhaus}, the map~$e\pi_{\F_p,Q}\circ \phi_p$ is strict. Let us now show that~$\phi_{p,Q}$ is strict. For this purpose, we take $u$ in $E(\F_p,Q)$ of degree~$n$, and we construct~$v$ in~$E(\Z_p,Q)$ of degree~$n$ such that $\phi_{p,Q}(v)=u$. 

There exists $\widetilde{u}$ in $E(\Z_p)$ of degree $n$, such that $e\pi_{\F_p,Q}\circ \phi_p(\widetilde{u})=u$. Let us denote by~$m$ the degree of~$v\coloneq e\pi_{\Z_p,Q}(\widetilde{u})$, so $m\geq n$. Since the diagram is commutative, then~$\phi_{p,Q}(v)=u$, thus $m\leq n$. Therefore~$m=n$, and so $\phi_{p,Q}$ is strictly filtered.
%from the commutativity of the previous diagram, we deduce that the map $\phi_{p,Q}\colon E(\Z_p,Q)\to E(\F_p,Q)$ is also surjective and strictly filtered (see \cite[Lemma $1.1.5$]{leoni2024zassenhaus}). 
%By commutativity, we also deduce (from \cite[Lemma $1.1.5$]{leoni2024zassenhaus}) that the map $\phi_p\colon I^0(\Z_p)\to I^0(\F_p)$ is strictly filtered. 
Consequently, all maps of the previous diagram are strictly filtered. We conclude with \cite[Lemma $1.3.6$]{leoni2024zassenhaus}.
\end{proof}

\begin{rema}
Lemma~\ref{strictfiltr} holds independently of the hypothesis that~$\Ll(\Z_p, Q)$ is torsion-free.
\end{rema}

Let us now study the kernel of $\grad(\phi_{p,Q})$. For this purpose, we need the following Lemmata:

\begin{lemm}\label{PBW torsion freeness}
The following conditions are equivalent:

    $(i)$ The $\Z_p$-graded module $\Ll(\Z_p, Q)$ is free.

    $(ii)$ The $\Z_p$-graded module $\E(\Z_p, Q)$ is free.
    \\Moreover, under these conditions, we have~$\U(\Ll(\Z_p, Q))\simeq \E(\Z_p, Q)$.
\end{lemm}

\begin{proof}
    The PBW Theorem (see \cite[Corollary $1.4$]{HAMZA2023172}) provides the implication $(i)\implies (ii)$. Observe that \cite[Theorem~$2.4$]{labute2006mild} shows $(ii)\implies (i)$.
    
    The last assertion is given by~\cite[Theorem~$1.3$]{hartlfox}. His proof can be easily adapted to the profinite case.
\end{proof}

\begin{lemm}\label{surj}
We have $\ker(\grad(\phi_{p,Q}))=p\E(\Z_p,Q)$, and so an isomorphism:
$$\E(\Z_p,Q)/p\E(\Z_p,Q)\simeq \E(\F_p,Q).$$
\end{lemm}

\begin{proof}
From~Lemma~\ref{PBW torsion freeness}, the~graded~$\Z_p$-module~$\E(\Z_p,Q)$ is torsion-free. Then from~\cite[Chapitre~I, $(2.3.12)$]{lazard1965groupes}, the module~$E(\Z_p,Q)$ is filtered-free over~$\Z_p$. Consequently, we deduce, from~\cite[Chapitre~I, $(2.3.12)$]{lazard1965groupes}, for every positive integer~$n$ a well-defined map:
$$pE_n(\Z_p,Q)\to p\E_n(\Z_p,Q);\quad pu\mapsto p\overline{u},$$
where $\overline{u}$ is the image of $u$ modulo $E_{n+1}(\Z_p,Q)$. Since $\E_n(\Z_p,Q)$ is torsion-free, then $p\overline{u}=0$ implies that $u$ is in $E_{n+1}(\Z_p,Q)$. Thus, the kernel of the previous map is $pE_{n+1}(\Z_p,Q)$. Since for every integer~$n$, $pE(\Z_p,Q)\cap E_n(\Z_p,Q)=pE_n(\Z_p,Q)$, we deduce:
$$\grad(pE(\Z_p,Q))\coloneq \bigoplus_{n\in \NN} pE_n(\Z_p,Q)/pE_{n+1}(\Z_p,Q)\simeq p\E(\Z_p,Q).$$
Furthermore, we have the following strict exact sequence:
$$0\to pE(\Z_p,Q)\to E(\Z_p,Q)\to^{\phi_{p,Q}} E(\F_p,Q)\to 0.$$
From \cite[Lemma $1.3.6$]{leoni2024zassenhaus}, we obtain the graded exact sequence:
$$0\to \grad(pE(\Z_p,Q))\simeq p\E(\Z_p,Q)\to \E(\Z_p,Q)\to^{\grad(\phi_{p,Q})} \E(\F_p,Q)\to 0.$$
\end{proof}

We conclude by showing that the image of the map~$\mu_p\coloneq \grad(\phi_p)|_{\I(\Z_p)}$ is~$\I(\F_p)$.

\begin{theo}\label{endanswermin}
We have a surjective map $\mu_p$ which makes the following diagram commutative:

\centering{
\begin{tikzcd}
0 \arrow[r] &  \I(\Z_p) \arrow[r, "\grad(\iota_{\Z_p,Q})", hook] \arrow[dd, two heads, "\mu_p"] &  \E(\Z_p) \arrow[r, "\grad(e\pi_{\Z_p,Q})", two heads] \arrow[dd, "\grad(\phi_p)", two heads] &   {\E(\Z_p,Q)} \arrow[r] \arrow[dd, "\grad(\phi_{p,Q})", two heads] &   0 \\
               &                                                              &                                                                   &                                                          &    \\
0 \arrow[r]   & \I(\F_p) \arrow[r,"\grad(\iota_{\F_p, Q})", hook]                                 &   \E(\F_p) \arrow[r, "\grad(e\pi_{\F_p,Q})", two heads]                                   & {\E(\F_p,Q)} \arrow[r]                                 &   0
\end{tikzcd}}
\end{theo}

\begin{proof}
\justifying
We define $\mu_p\coloneq \grad(\phi_p)\circ \grad(\iota_{\Z_p, Q})$ and compute its image in $\E(\F_p)$. This image is isomorphic to 
$$\left( \I(\Z_p)+p\E(\Z_p)\right)/p\E(\Z_p)\simeq \I(\Z_p)/\left(\I(\Z_p)\cap p\E(\Z_p)\right).$$
From Lemma~\ref{PBW torsion freeness}, the~$\Z_p$-module~$\E(\Z_p,Q)$ is torsion-free. Then~$\I(\Z_p)\cap p\E(\Z_p)=p\I(\Z_p)$. Thus the image of $\mu_p$ is $\I(\Z_p)/p\I(\Z_p)$. 

Let us now show that $\I(\Z_p)/p\I(\Z_p)\simeq \I(\F_p)$. From Lemma \ref{surj}, we have the chain of isomorphisms: 
$$\E(\F_p)/\I(\F_p)\simeq \E(\F_p,Q)\simeq \E(\Z_p,Q)/p\E(\Z_p,Q)\simeq \left(\E(\Z_p)/\I(\Z_p)\right)/p\left(\E(\Z_p)/\I(\Z_p)\right).$$
Similarly, since $\E(\Z_p,Q)$ is torsion-free, we deduce from the canonical map $p\E(\Z_p)\to p\E(\Z_p,Q), pu\mapsto p \grad(e\pi_{\Z_p,Q})(u)$, the isomorphism:
$$p\E(\Z_p,Q)\simeq p\E(\Z_p)/p\I(\Z_p).$$
Thus, we have 
$$\E(\F_p)/\I(\F_p)\simeq \left( \E(\Z_p)/p\E(\Z_p)\right)/\left(\I(\Z_p)/p\I(\Z_p)\right).$$ 
So $\I(\F_p)\simeq \I(\Z_p)/p\I(\Z_p)$. Then the image of $\mu_p$ is $\I(\F_p)$ and the kernel of $\mu_p$ is~$p\I(\Z_p)$.
\end{proof}

\begin{proof}[Proof Theorem \ref{answerminac2}]
Let us write $i_n(\AA)\coloneq \rk_{\AA}\I_n(\AA)$. From Theorem \ref{endanswermin}, we have :
$$c_n(\Z_p,F)=c_n(\F_p,F), \quad \text{and} \quad c_n(\Z_p,Q)\geq c_n(\F_p,Q), \quad \text{and} \quad i_n(\Z_p)\geq i_n(\F_p).$$ From $c_n(\AA,F)=i_n(\AA)+c_n(\AA,Q)$, it follows that~$c_n(\Z_p,Q)=c_n(\F_p,Q)$.
\end{proof}

\begin{rema}\label{kernel rema}
The proof of Theorem~\ref{answerminac2} also shows that $i_n(\F_p)=i_n(\Z_p)$ for every positive integer $n$.
\end{rema}

\subsection{Graph products of graded algebras} 
We aim here to study torsion on graph products of algebras using techniques from the previous section. Let us start with the following lemma:

\begin{lemm}\label{PBW application bis}
Assume that $G$ satisfies $(\ast_\AA)$, then the $\AA$-universal envelope of~$\Gamma \Ll (\AA,G)$, denoted by $\mathcal{U}(\Gamma \Ll (\AA,G))$ is isomorphic to $\Gamma \E(\AA, G)$.
\end{lemm}

\begin{proof}
The universal envelope functor $\U$ from the category of locally finite graded $\AA$-Lie algebras to the category of $\AA$-associative algebras is a left adjoint, so preserves colimits. Furthermore, it preserves products (see \cite[§2, Proposition 2]{bourbaki2007groupes}). Thus $\U$ preserve graph products (see \cite[§2]{panovpseries}), and so
$$\U(\Gamma \Ll(\AA,G))\simeq \Gamma \U(\Ll(\AA,G),$$
where $\Ll(\AA,G):=\{\Ll(\AA,{}_1G),\dots, \Ll(\AA,{}_kG)\}$. Since $G$ satisfies $(\ast_\AA)$, we observe (see for instance \cite[Corollary $1.4$]{HAMZA2023172}) that $\U(\Ll(\AA,{}_iG))\simeq \E(\AA, {}_iG)$. Thus
$$\U(\Gamma \Ll(\AA,G))\simeq \Gamma \E(\AA, G).$$
%We observe that $\E({}_iG)$ is the universal envelope of $\Ll({}_iG)$. Thus we infer a map $\mathcal{U}(\Gamma \Ll(G))\to \Gamma \E(G)$. Furthermore, we have maps $\Ll({}_iG)\to \E({}_iG)$, which gives us a map $\Gamma \Ll(G)\to \Gamma \E(G)$, which induces a map $\mathcal{U}(\Gamma\Ll(G))\to \Gamma \E(G)$. 
%Since the $\AA$-Lie algebra $\Gamma \Ll(\AA,G)$ is a free module over $\AA$, the PBW Theorem (see for instance \cite[Corollary $1.4$]{HAMZA2023172}) shows that the algebras $\mathcal{U}(\Gamma \Ll(\AA, G))$ and $\Gamma \E(\AA, G)$ have the same presentation. Thus, we have an isomorphism:
%$$\mathcal{U}(\Gamma \Ll(\AA, G))\simeq \Gamma \E(\AA, G).$$
\end{proof}

We now state the starting point of the proof of Theorem~\ref{main result}.

\begin{theo}\label{mainresultfree}
Let~$G$ be a family of finitely generated pro-$p$ groups satisfying~$(\ast_\AA)$. Let $\Gamma^f$ be the free graph on $k$ vertices. 

We have:
$\Ll(\AA, \Gamma^f G)\simeq \Gamma^f \Ll(\AA,G)$ and $\E(\AA, \Gamma^f G)\simeq \Gamma^f \E(\AA,G)$. Furthermore:
$$\U(\Ll(\Z_p, \Gamma^f G))\simeq \E(\Z_p, \Gamma^f G).$$
\end{theo}

\begin{proof}
The case~$\AA\coloneq \F_p$ was proved by~\cite[Theorem~$1.2$]{Lichtman}. Indeed only the abstract case was considered, but we can conclude for the pro-$p$ case using
%have for every finitely generated abstract group $Q$, $\Ll(\F_p,Q)\simeq \Ll(\F_p,\widehat{Q}^p)$ and $\widehat{Q_1\coprod Q_2}^p\simeq \widehat{Q_1}^p\coprod \widehat{Q_2}^p$. 
Lemma~\ref{PBW application bis} and \cite[Page~$312$]{DDMS} ($c_n(\F_p,Q)=\rk_{\F_p} I^0_n(G)/I^0_{n+1}(G)$ where~$I_n^0(G)$ is the $n$-th power of the augmentation ideal of the group algebra of~$G$ over~$\F_p$).

The case~$\AA\coloneq \Z_p$ follows from~\cite[Corollary]{Smel} (originally, only the abstract group case was considered, but an adaptation of this proof to the pro-$p$ case is straightforward).
Indeed, we have $\Ll(\Z_p, \Gamma^f G)\simeq \Gamma^f\Ll(\Z_p,G)$ which are both torsion-free. Furthermore, since every pro-$p$ group~${}_iG$ is finitely generated, we observe that~$\bigcap_{n\in \NN}G_n(\Z_p)\subset \bigcap_{n\in \NN}G_n(\F_p)=\{1\}$. Finally, from \cite[Corollary~$1.4$]{HAMZA2023172}, we conclude that $\U(\Ll(\Z_p, \Gamma^f G))\simeq \E(\Z_p, \Gamma^f G)$.
\end{proof}

The rest of this section aims to illustrate techniques used in the proof of Theorem~\ref{endanswermin}. Let $G\coloneq \{ {}_1G,\dots, {}_kG\}$ be a family of finitely generated pro-$p$ groups satisfying~$(\ast_{\Z_p})$. We introduce~$\E(\Z_p, G)\coloneq \{\E(\Z_p, {}_1G), \dots, \E(\Z_p, {}_kG)\}$. We show that the algebra~$\Gamma\E(\Z_p,G)$ is torsion-free.
 
\begin{prop}\label{ast and graph}
    If the family $G$ satisfies the condition $(\ast_{\Z_p})$  then $\Gamma \E(\Z_p,G)$ and~$\Gamma \Ll(\Z_p, G)$ are free over~$\Z_p$. 
\end{prop}

\begin{proof}
From Lemma~\ref{PBW torsion freeness}, it is sufficient to show that~$\Gamma \E(\Z_p,G)$ is torsion-free. Recall that $\E(\AA,{}_iG)\coloneq {}_i\E(\AA)/\I(\AA,{}_iR)$, and we define $\Gamma^f\I(\AA,R)$ the two sided ideal in $\E(\AA)$ generated by $\I(\AA,{}_iR)$ for every $1\leq i\leq k$, and $\I(\AA, \Gamma)$ the ideal in $\E$ generated by $[{}_iX_u,{}_jX_v]$ with $\{i,j\}\in \bE$. Observe that $\Gamma \I(\AA,R)\coloneq \Gamma^f\I(\AA,R)+\I(\AA,\Gamma)$ is the kernel of the map $\E(\AA) \to \Gamma \E(\AA,G)$. Take $u$ in $\E_n(\Z_p)$. We show that if $p u$ is in~$(\Gamma \I)_n(\Z_p,R)$, then~$u$ is also in $(\Gamma \I)_n(\Z_p,R)$.
    
    From Theorem~\ref{mainresultfree}, it follows that that the algebra $\E(\Z_p, \Gamma^f G)\simeq \coprod_{i=1}^k\E(\Z_p, {}_iG)$ is torsion-free. Denote by~$G_\Gamma$ a Right Angled Artin group presented by relations $[{}_ix_u,{}_jx_v]$ with $\{i,j\}\in \bE$. Then, the algebra $\E(\Z_p, \Gamma)\coloneq \E(\Z_p)/\I(\Z_p, \Gamma) \simeq \U(\Ll(\Z_p, G_\Gamma))$ is torsion-free (see for instance~\cite{wade2016lower} or~\cite{bartholdi2020right} for further references). Thus, from Theorem~\ref{endanswermin}, we deduce surjections
    $$\mu_{p,\Gamma^f G}\colon \Gamma^f \I(\Z_p,R)\to \Gamma^f \I(\F_p,R) , \quad \text{and } \mu_{p,\Gamma}\colon \I(\Z_p, \Gamma)\to \I(\F_p, \Gamma).$$ %which gives us a surjection~$\mu_{p,\Gamma G}\colon \Gamma \I(\Z_p, R)\to \Gamma \I(\F_p, R)$. 
   We define $\eta(\F_p)$ (resp. $\delta(\F_p)$) a $\F_p$ basis of $\Gamma_n^f\I(\F_p, R)$ (resp. $\I_n(\F_p, \Gamma)$). Using $\mu_{p,\Gamma^f G}$, $\mu_{p,\Gamma}$ and Remark~\ref{kernel rema}, we lift these basis to $\eta(\Z_p)$ (resp. $\delta(\Z_p)$) which is a basis of~$(\Gamma^f\I)_n(\Z_p, R)$ (resp.~$\I_n(\Z_p, \Gamma)$). Recall that $\E(\Z_p, \Gamma^f G)$ and $\E(\Z_p, G_\Gamma)$ are torsion-free. Thus if $p u$ is in $(\Gamma^f\I)_n(\Z_p,R)$ (resp. $\I_n(\Z_p,\Gamma)$) then $u$ is in $(\Gamma^f\I)_n(\Z_p,R)$ (resp.~$\I_n(\Z_p,\Gamma)$).
    
    We observe that $\J(\Z_p)\coloneq \Gamma^f\I(\Z_p,R) \cap \I(\Z_p,\Gamma)$ is an ideal in $\E(\Z_p)$ and $\E(\Z_p)/\J(\Z_p)$ is torsion-free over $\Z_p$. Thus we rewrite $\eta(\Z_p):=\{\eta'(\Z_p),\beta(\Z_p)\}$ (resp. $\delta(\Z_p):=\{\delta'(\Z_p), \beta(\Z_p)\}$) a $\Z_p$-basis of $(\Gamma^f\I)_n(\Z_p, \Rr)$ (resp. $\I_n(\Z_p, \Gamma)$) where 
 $\beta(\Z_p)$ is a $\Z_p$-basis of $(\Gamma^f\I)_n(\Z_p, \Rr)\cap \I_n(\Z_p, \Gamma)$.
%        \item $\eta(\F_p):=\mu_{p, \Gamma^f G}(\eta(\Z_p))$ (resp. $\delta(\F_p):=\mu_{p, \Gamma}(\delta(\Z_p))$) is a $\F_p$- basis of $(\Gamma^f\I)_n(\F_p, \Rr)$ (resp. $\I_n(\F_p, \Gamma)$).

In particular, the family $\eta(\Z_p)\cup \delta(\Z_p)\coloneq \{\eta'(\Z_p), \delta'(\Z_p), \beta(\Z_p)\}$ is $\Z_p$-free, and generates~$(\Gamma\I)_n(\Z_p,R)$. Furthermore, if $pu$ is in $(\Gamma\I)_n(\Z_p,R)$, then we write:
     $$pu\coloneq \sum {\alpha_i \eta_i'(\Z_p)+\alpha_j \delta_j'(\Z_p)+\alpha_l \beta_l(\Z_p).}$$
     Which provides, mod $p$ the relation:
     $$0=\sum{ \overline{\alpha_i}\eta_i'(\F_p)+ \overline{\alpha_j}\delta_j'(\F_p)+\overline{\alpha_l}\beta_l(\F_p)},$$
where $\overline{\alpha}$ is the image of $\alpha$ modulo $p\Z_p$.     Therefore, $\alpha_i$, $\alpha_j$ and $\alpha_l$ are all in $p\Z_p$, thus we write:
     $$p u=p  \sum{\left( \alpha_i'\eta_i'(\Z_p)+ \alpha_j'\delta_j'(\Z_p)+\alpha_l'\beta_l(\Z_p) \right)}.$$
     So $u=\sum_{i,j,l}{ \left(\alpha_i'\eta_i'(\Z_p)+ \alpha_j'\delta_j'(\Z_p)+\alpha_l'\beta_l(\Z_p)\right)}$ is an element in $\Gamma\I(\Z_p,R)$.
     \end{proof}

\section{Proof of Theorem \ref{main result}}
%Let~$Q$ be a finitely generated pro-$p$ group. We define:
%$$\Ll(\Q_p, Q)\coloneq \Ll(\Z_p,Q)\otimes_{\Z_p} \Q_p, \quad \text{and } \E(\Q_p, Q)\coloneq \U(\Ll(\Z_p, Q))\otimes_{\Z_p}\Q_p.$$

In order to prove Theorem~\ref{main result}, we need some results on amalgamated products and related filtrations. Leoni~\cite{leoni2024zassenhaus} studied the Zassenhaus filtration of amalgamated products under a natural condition. As we see in this section, his result can be extended for the lower central series. 

Let us denote by~$\fq$ the fraction field of~$\AA$.
We define~$E(\Q_p, Q)$ the completed group algebra of~$Q$ over~$\Q_p$. For every integer~$n$, we denote by~$E_n(\Q_p, Q)$ the~$n$-th power of the augmentation ideal of~$E(\Q_p, Q)$. The algebra~$E(\Q_p, Q)$ is filtered by the family~$\{E_n(\Q_p, Q)\}_{n\in \NN}$.
We consider~$\E_n(\Q_p, Q)\coloneq \grad(E(\Q_p, Q))\coloneq \bigoplus_{n\in \NN}E_n(\Q_p, Q)/E_{n+1}(\Q_p, Q)$. Let us denote by~$\fq\lbrack Q \rbrack$ the group algebra of~$Q$ over~$\fq$, filtered by the~$n$-th power of the augmentation ideal. Since~$Q$ is finitely generated, we observe that~$\E(\fq, Q)\simeq \grad(\fq\lbrack Q \rbrack)$.
Let us also define~$\Ll(\Q_p, Q)\coloneq \Ll(\Z_p, Q)\otimes_{\Z_p}\Q_p$. 
As showed by Quillen~\cite[Theorem]{QUILLEN1968411}, the algebra~$\E(\fq, Q)$ is the universal envelope of~$\Ll(\fq, Q)$.

\begin{theo}[Quillen]
If~$Q$ is a finitely generated pro-$p$ group, we have a graded isomorphism:
$$\U(\Ll(\fq, Q))\simeq \E(\fq, Q).$$
\end{theo}

Let~$Q_1$, $Q_2$ and~$H$ be finitely generated pro-$p$ groups. We assume that we have morphisms~$\lambda_1\colon H \to Q_1$ and~$\lambda_2\colon H \to Q_2$. From~$\lambda_1$ and~$\lambda_2$, we infer morphisms~$\grad(\fq, \lambda_1)\colon \Ll(\fq, H)\to \Ll(\fq, Q_1)$ and~$\grad(\fq, \lambda_2)\colon \Ll(\fq, H)\to \Ll(\fq, Q_2)$. Therefore, we define the amalgamated product:
$$Q\coloneq Q_1\coprod_{H} Q_2.$$
We also have an amalgamated product of~$\Ll(\fq, Q_1)$ and~$\Ll(\fq, Q_2)$ over~$\Ll(\fq, H)$, that we denote by~$\Ll(\fq, Q_1)\coprod_{\Ll(\fq, H)}\Ll(\fq, Q_2)$.
We refer to~\cite{RibesZal} and~\cite[Chapter~$4$]{leoni2024zassenhaus} for more details on amalgamated products.

We have a natural morphism~$\Ll(\fq, Q_1)\coprod_{\Ll(\fq, H)}\Ll(\fq, Q_2)\to \Ll(\fq, Q)$. Leoni~\cite[Theorem~$7$, Chapter~$4$]{leoni2024zassenhaus} showed that the previous morphism is an isomorphism when~$\fq\coloneq \F_p$ and~$\grad(\F_p, \lambda_i)$ is injective for~$i\in \{1,2\}$. This motivates the following definition:  

\begin{defi}
Let~$H$ and~$Q$ be finitely generated pro-$p$ groups. Assume that we have an injective morphism~$\lambda\colon H \to Q$. We say that~$H$ is strictly~$\fq$-embedded in~$Q$ if the morphism~$\grad(\fq, \lambda)\colon \Ll(\fq, H) \to \Ll(\fq, Q)$ is injective.
\end{defi}

 %We define the filtration:
%$$Q_n(\Q_p)\coloneq \{q\in Q, q-1\in E_n(\Q_p, Q).$$
%This allows us to define the~$\Q_p$-graded algebra:
%$$\Ll(\Q_p, Q)\coloneq \bigoplus_{n\in \NN} Q_n(\Q_p)/Q_{n+1}(\Q_p).$$

Leoni's proof of~\cite[Theorem~$7$, Chapter~$4$]{leoni2024zassenhaus} yields the following result:

\begin{theo}[Leoni]\label{leoni result}
Let~$Q_1$, $Q_2$ and~$H$ be finitely generated pro-$p$ groups such that~$H$ is strictly~$\fq$-embedded in~$Q_1$ and in~$Q_2$. Then the graded morphism
$$\Ll(\fq, Q_1)\coprod_{\Ll(\fq, H)} \Ll(\fq, Q_2)\to \Ll(\fq, Q_1\coprod_{H}Q_2)$$
is an isomorphism.
\end{theo}

\subsection{Field situation}
Let~$G$ be a family of pro-$p$ group satisfying~$(\ast_\AA)$ and~$\Gamma$ be an undirected graph. As before, the family~$G$ defines a family of graded $\fq$-Lie algebra:
$$\Ll(\fq, G)\coloneq \{\Ll(\fq, {}_1G),\dots , \Ll(\fq, {}_kG)\}.$$
We naturally define~$\Gamma \Ll(\fq, G)$.
The goal of this subsection is to show the following result:
\begin{theo}\label{iso finite-case}
We have an isomorphism
$$\Ll(\fq, \Gamma G)\simeq \Gamma \Ll(\fq,G).$$
\end{theo}

The proof will be done by induction on the number of vertices of~$\Gamma$. Let~${}'\Gamma\coloneq ({}'\bX, {}'\bE)$ be a subgraph of~$\Gamma$ and~${}'G$ be the subfamily of~$G$ indexed by~${}'\bX$. We have a canonical map~$\iota_{'\Gamma/\Gamma}\colon {}'\Gamma {}'G \to \Gamma G$ which maps~${}_ix_j$ to~${}_ix_j$ for every~$i\in {}'\bX$ and~$1\leq j \leq {}_id$.

\begin{lemm}\label{injection of graphs}
The map~$\iota_{{}'\Gamma/\Gamma}\colon {}'\Gamma {}'G \to \Gamma G$ is injective. Furthermore, the associated map~$\grad(\AA, \iota_{{}'\Gamma/\Gamma})$ is injective. Similarly, we have an injection~$'\Gamma \E(\F_p, {}'G)\to \Gamma \E(\F_p, G)$.
%Furthermore, for every integer~$n$, the canonical map~$\iota_{'\Gamma/\Gamma,n, \AA}\colon ({}'\Gamma '{}G)_n(\AA)/({}'\Gamma {}'G)_{n+1}(\AA) \to (\Gamma G)_n(\AA)/(\Gamma G)_{n+1}(\AA)$ is injective.
\end{lemm}

\begin{proof}
For every~$i\in {}'\bX$ and~$1\leq j \leq {}_id$, we observe that:
$$\iota_{\Gamma /{}'\Gamma} \circ \iota_{{}'\Gamma/\Gamma}({}_ix_j)={}_ix_j, \quad \text{and} \quad \grad(\AA, \iota_{\Gamma /{}'\Gamma}) \circ \grad(\AA, \iota_{{}'\Gamma/\Gamma})({}_iX_j)={}_iX_j.$$
Thus, we infer the following commutative diagram:

\begin{center}
% https://tikzcd.yichuanshen.de/#N4Igdg9gJgpgziAXAbVABwnAlgFyxMJZABgBoAmAXVJADcBDAGwFcYkRgBfAcgB1eA4vQC2w+gAIAFF24CAlCE6l0mXPkIoyxanSat2-IaIkDFykBmx4CRcqW00GLNog49DIsVJnyzKq+pEACwUOk76rvwAMoyS-ACC8aTiMh7GKTy+Sv5qNighDrrOBrwxcbyJyWleWeaWuRrIAGz2YXouINGxCUkZfIKeEj4KnDowUADm8ESgAGYAThDCSGQgOBBIAIyO7SX4OPQA+sD9RmIA9NX0nH4gC0tbNOtIdkURnbz7R8BX56kDxhu2Tui2WiFWz0Qr3CHSwUFu9zBAGYnhtECE3h1+BN5vQoOVKuJ+F9jqdBpcAWJOCNzIikBjIQBWHbFSK8HF4gm9YkQA7HX7-M7XBGgpAotZo5mY9hwxSUThAA
\begin{tikzcd}
\Gamma G \arrow[rr, "\iota_{\Gamma/{}'\Gamma}"]                          &  & {}'\Gamma ({}'G) &  & {\Ll(\AA, \Gamma G)} \arrow[rr, "{\grad(\AA, \iota_{\Gamma/{}'\Gamma})}"]                         &  & {\Ll(\AA, {}'\Gamma {}'G)} \\
                                                                         &  &                  &  &                                                                                                  &  &                            \\
{}'\Gamma ({}'G) \arrow[uu, "\iota_{'\Gamma/\Gamma}"] \arrow[rruu, "id"] &  &                  &  & {\Ll(\AA, {}'\Gamma {}'G)} \arrow[uu, "{\grad(\AA, \iota_{'\Gamma/\Gamma})}"] \arrow[rruu, "id"] &  &                           
\end{tikzcd}
\end{center}

\justifying
Therefore, we conclude that~$\iota_{'\Gamma/\Gamma}$ and~$\iota_{'\Gamma/\Gamma, n ,\AA}$ are both injective.
Considering similar diagrams, we also show that we have an injection~${}'\Gamma \E({}'G)\to \Gamma \E(G)$.
\end{proof}

\begin{coro}\label{strict embed}
For every subgraph~${}'\Gamma$ of~$\Gamma$, the pro-$p$ group~${}'\Gamma {}'G$ is strictly~$\fq$-embedded in~$\Gamma G$.
\end{coro}

\begin{proof}
From Lemma~\ref{injection of graphs}, we have an injection:
$$\grad(\AA, \iota_{{}'\Gamma/\Gamma})\colon \Ll(\AA, {}'\Gamma {}'G)\to \Ll(\AA, \Gamma).$$
We conclude tensorizing by~$\fq$ over~$\AA$. 
\end{proof}

\begin{proof}[Proof of Theorem~\ref{iso finite-case}]
We argue by induction on~$k$, the number of vertices of~$\Gamma$. If~$k=1$. This is fine. Let us now assume that~$k>1$. We have the following alternative.

$(i)$ If the graph~$\Gamma$ is complete, then we deduce that~$\Gamma G\simeq \prod_{i=1}^k {}_iG$. Thus~$\Ll(\AA, \Gamma G)\simeq \bigoplus_{i=1}^k \Ll(\AA, {}_iG)\simeq \Gamma \Ll(\AA, G)$, and:
$$\Ll(\fq, \Gamma G)\simeq \Ll(\AA, \Gamma G)\bigotimes_{\AA}\fq\simeq \left( \bigoplus_{i=1}^k \Ll(\AA, {}_iG)\right)\bigotimes_{\AA}\fq \simeq \bigoplus_{i=1}^k \Ll(\fq, {}_iG)\simeq \Gamma \Ll(\fq, G).$$

$(ii)$ We assume that~$\Gamma$ is not complete. Then there exists two vertices~$a$ and~$b$ which are not connected by an edge. We define~${}^a\Gamma$, ${}^b\Gamma$ and ${}^c\Gamma$ the subgraphs of~$\Gamma$ with vertices~$\bX\setminus \{a\}$, $\bX\setminus \{b\}$ and~$\bX\setminus \{a,b\}$. Consider~${}^aG\coloneq G\setminus \{{}_{a}G\}$,  ${}^bG \coloneq G\setminus \{{}_{b}G\}$ and~${}^cG \coloneq G\setminus \{ {}_{a}G, {}_{b}G\}$. Note, from Lemma~\ref{injection of graphs} and Proposition~\cite[Proposition~$9.2.1$]{RibesZal}, that we have the following isomorphism:
$$\Gamma G \simeq ({}^a\Gamma{}^aG) \coprod_{({}^c\Gamma {}^cG)} ({}^b\Gamma {}^bG).$$
A similar isomorphism was observed by Green~\cite[Proof of the Lemma~$3.20$]{ruthgraph}. We also recall, from the induction hypothesis, that:
$$\Ll(\fq,{}^i\Gamma{}^iG)\simeq {}^i\Gamma \Ll(\fq,{}^iG),\quad \text{ where } i\in \{a,b,c\}.$$
Furthermore, from Corollary~\ref{strict embed}, the pro-$p$ group~${}^c\Gamma {}^c G$ is strictly~$\fq$-embedded in~${}^a\Gamma {}^a G$ and in~${}^b\Gamma {}^b G$. Thus, from~Theorem~\ref{leoni result}, we conclude
$$\Ll(\fq,\Gamma G)\simeq \Ll(\fq,{}^a\Gamma{}^aG) \coprod_{\Ll(\fq, {}^c\Gamma {}^cG)} \Ll(\fq,{}^b\Gamma {}^bG)\simeq \Gamma \Ll(\fq, G).$$
\end{proof}

\subsection{The case~$\AA=\Z_p$}
To conclude the proof of Theorem~\ref{main result}, we show the following result:

\begin{theo}\label{main theo Zp}
Assume that~$G$ satisfies the condition~$(\ast_{\Z_p})$, then~$\Ll(\Z_p, \Gamma G)$ is torsion-free, and we have:
$$\Ll(\Z_p, \Gamma G)\simeq \Gamma \Ll(\Z_p, G).$$
\end{theo}

We introduce the following notations.

$\bullet$ If~$\Ll$ is a graded Lie algebra over~$\Z_p$, we define~$a_n(\Ll)=\rk_{\Z_p}\Ll_n$.

$\bullet$ If~$\Ll$ is a graded Lie algebra over~$\Q_p$, we define~$a_n(\Ll)=\dim_{\Q_p}\Ll_n$.

\begin{lemm}\label{torsion and Q}
Let~$\Ll$ be a graded-Lie algebra over~$\Z_p$. The following assertions are equivalent:

$\bullet$ For every integer~$n$, we have~$a_n(\Ll\otimes_{\Z_p}\Q_p)=a_n(\Ll)$. 

$\bullet$ The~$\Z_p$-module $\Ll$ is torsion-free.
\end{lemm}

\begin{proof}
Let us write~$\Ll=\Ll_{free}\bigoplus \Ll_{tor}$, where~$\Ll_{free}$ is the free part of~$\Ll$ and~$\Ll_{tors}$ is the torsion part. The module~$\Ll$ is torsion-free if and only if~$\Ll_{tors}$ is trivial. We conclude the proof using the following isomorphism:
$$\Ll \otimes_{\Z_p} \Q_p\simeq \Ll_{free}\otimes_{\Z_p}\Q_p.$$
\end{proof}

Let~$\Ll(\AA, G)$ be the family~$\{\Ll(\AA, {}_1G, \dots, \Ll(\AA, {}_kG)\}$. We recall that~$\Gamma \Ll(\AA, G)$ is the graph product of the family~$\Ll(\AA, G)$ by~$\Gamma$. This is a quotient of~$\Ll(\AA)$, the free graded~$\grad(\AA)$-Lie algebra on~${}_iX_j$ for~$1\leq i \leq k$ and~$1\leq j \leq {}_id$.
\begin{lemm}\label{Gamma epi}
We have a natural surjective morphism of $\AA$-Lie algebras $\alpha \colon \Gamma \Ll(\AA,G)\twoheadrightarrow \Ll(\AA,\Gamma G)$, where $\Ll(\AA,G):=\{\Ll(\AA,{}_1G),\dots , \Ll(\AA,{}_nG)\}$.

Therefore, we have a surjection of graded-Lie algebras:
$$\alpha\otimes_{\AA}\fq \colon \Gamma \Ll(\fq,G)\twoheadrightarrow \Ll(\fq,\Gamma G),$$ where $\Ll(\fq,G):=\{\Ll(\fq,{}_1G),\dots , \Ll(\fq,{}_kG)\}$.
\end{lemm}

\begin{proof}
Consider the presentation
$$1\to \Gamma R \to F \to \Gamma G\to 1,$$
where $F$ is the free pro-$p$ group on~${}_iX_j$ for~$1\leq i \leq k$ and~$1\leq j \leq {}_id$, and $\Gamma R$ is the closed normal subgroup of $F$ generated by ${}_iR$ with~$1\leq i \leq k$ and~$[{}_uX_a, {}_vX_b]$ with~$\{u,v\}\in \bE$ and~$1\leq a \leq {}_ud$ and~$1\leq b \leq {}_vd$. 
We define $\Rr(\Gamma)$ the kernel of the following epimorphism:
$$\widetilde{\alpha}\colon \Ll(\AA)\twoheadrightarrow \Ll(\AA, \Gamma G).$$
Observe that~$\Rr(\Gamma)$ is isomorphic to~$\bigoplus_n \Gamma R \cap F_n(\AA)/ \Gamma R \cap F_{n+1}(\AA)$.
Let us write~$\Ll(\AA,{}_iG)\coloneq {}_i\Ll/{}_i\Rr$, where ${}_i\Ll$ is a free $\AA$-Lie algebra on ${}_iX_1,\dots , {}_iX_{{}_id}$, and ${}_i\Rr\simeq \bigoplus_n \left( {}_iR \cap {}_iF_n(\AA)\right)/\left({}_iR \cap {}_iF_{n+1}(\AA)\right)$ (the isomorphism is given by Magnus). Since for every positive integer $n$, ${}_iF_n(\AA)\subset F_n(\AA)$, then ${}_i\Rr$ embeds into a subset of $\Ll(\AA)$. In particular, since ${}_iR\subset \Gamma R$, we infer that~${}_i\Rr\subset \Rr(\Gamma)$.

Let $\Gamma \Rr$ be the $\AA$-Lie ideal in $\Ll(\AA)$ generated by ${}_i\Rr$ with $1\leq i \leq k$ and the family~$\gamma \coloneq \{ [{}_uX_a,{}_vX_b]\}$ with~$\{u,v\}\in \bE$ and~$1\leq a\leq {}_ud$ and $1\leq b \leq {}_vd$. This is the kernel of the canonical morphism~$\Ll(\AA) \twoheadrightarrow \Gamma \Ll(\AA, G)$. To conclude, we easily observe that $\gamma\subset \Rr(\Gamma)$. Thus $\Gamma \Rr\subset \Rr(\Gamma)$, and so $\widetilde{\alpha}$ induces a surjection
$$\alpha\colon \Gamma \Ll(\AA,G)\to \Ll(\AA,\Gamma G).$$
\end{proof}

%We introduce $\Ll_\nu(\AA,\Gamma G):=\bigoplus_n \nu_n(\AA,\Gamma G)/\nu_{n+1}(\AA,\Gamma G)$. By the definition of $\nu$, we obtain a $\AA$-Lie graded morphism $\zeta \colon \Ll_\nu(\AA, \Gamma G)\to \grad_{\nu}\Gamma E(\AA,G)$. Moreover $\nu$ is a $\AA$-filtration on $\Gamma G$ (since ${\rm evstd}_G$ is a morphism of $\AA$-algebras). So we infer a morphism of $\grad(\AA)$-Lie algebras $\beta\colon \Ll(\AA,\Gamma G)\to \Ll_\nu(\AA,\Gamma G)$.

%Furthermore the map $\psi_{\Gamma G} \colon (\Gamma E(\AA,G), \nu) \to (\Gamma C(\AA), \{\Gamma_n C(\AA)\}_n)$ is a strict surjective filtered morphism (by the definition of $\nu$). Thus we have a graded epimorphism:
%$$\grad_\nu(\psi_{\Gamma G})\colon \grad_\nu(E(\AA,\Gamma G))\to \Gamma \CC(\AA).$$

\begin{proof}[Proof of Theorem~\ref{main theo Zp}]
Since~$G$ satisfies the condition~$(\ast_{\Z_p})$, then, from Proposition~\ref{ast and graph}, the Lie algebra~$\Gamma \Ll(\Z_p, G)$ is torsion-free. From Lemmata~\ref{torsion and Q} and~\ref{Gamma epi}, we obtain the following inequalities:
\begin{multline*}
a_n(\Gamma \Ll(\Q_p, G))=a_n(\Gamma \Ll(\Z_p, G))\geq a_n(\Ll(\Z_p, \Gamma G))\geq a_n(\Ll(\Q_p, \Gamma G))=a_n(\Gamma \Ll(\Q_p, G)),
\end{multline*}
the last equality is given by~Theorem~\ref{iso finite-case}.

This implies that~$a_n(\Ll(\Z_p, \Gamma G))=a_n(\Ll(\Q_p, \Gamma G))$. Thus~$\Ll(\Z_p, \Gamma G)$ is torsion-free, and the morphism~$\alpha\otimes_{\Z_p}\Q_p\colon \Gamma \Ll(\Q_p, G)\to \Ll(\Q_p, \Gamma G)$, given in Lemma~\ref{Gamma epi}, is an isomorphism. Consequently, the surjective morphism~$\alpha\colon \Gamma \Ll(\Z_p,G)\twoheadrightarrow \Ll(\Z_p, \Gamma G)$ is also an isomorphism.
\end{proof}

\begin{coro}\label{iso enveloppe}
Assume that the family~$G$ satisfies~$(\ast_\AA)$.Then~$\E(\AA, \Gamma G)$ is torsion-free. Furthermore, if we define~$\E(\AA, G)\coloneq \{\E(\AA, {}_1G), \dots, \E(\AA, {}_k G)\}$, we infer
$$\E(\AA, \Gamma G)\simeq \Gamma \E(\AA, G).$$
\end{coro}

\begin{proof}
From Theorem~\ref{main theo Zp}, the $\AA$-module~$\Ll(\AA, \Gamma G)$ is torsion-free. Thus, we infer the following argument. On the one hand, from Lemma~\ref{PBW torsion freeness}, we have~$\U(\Ll(\AA, \Gamma G))\simeq \E(\AA, \Gamma G)$. On the other hand, from Lemma~\ref{PBW application bis}, we have~$\U(\Gamma \Ll(\AA, G))\simeq \Gamma \E(\AA, G)$. Therefore, we have:
$$\E(\AA, \Gamma G)\simeq \Gamma \E(\AA, G).$$
\end{proof}

\section{Computation of gocha series and Koszulity}
In this section we only consider finitely presented pro-$p$ groups. Let $G$ be a family of pro-$p$ groups. Except Corollary~\ref{app Koszul}, we assume that $\AA=\F_p$ and that for every $1\leq i \leq k$, the graded algebra $\E({}_iG)$ is Koszul. We are mostly inspired by \cite[Subsection~$4.1$]{bartholdi2020right}.
\subsection{Graph product of Koszul algebras}
We show the following result:
\begin{theo}\label{Koszul graph}
   Assume that for every $1\leq i \leq k$, the graded algebra $\E({}_iG)$ is Koszul, then the graded algebra $\E(\Gamma G)$ is also Koszul.
\end{theo}

\begin{proof}
We proceed by induction on~$k$, the number of vertices of~$\Gamma$.

If~$k=1$, there is nothing to show.

If~$k>1$, we distinguish the two following cases.

$(i)$ If~$\Gamma$ is a complete graph, we deduce that~$\E(\Gamma G)\simeq \prod_{i=1}^k \E({}_iG)$. Since for every~$1\leq i \leq k$ the algebra~$\E({}_iG)$ is Koszul, we infer by~\cite[Chapitre~$3$, Corollary~$1.2$]{polishchuk2005quadratic} that~$\E(\Gamma G)$ is Koszul.

$(ii)$ Let us assume that~$\Gamma$ is not a complete graph.
Then following the notations and arguments from the proof of Theorem~\ref{iso finite-case}~$(ii)$, 
%and the Normal Form Theorem from Green~\cite[Proof of the Lemma~$3.20$]{ruthgraph} (which also works for algebras)
we have the following isomorphisms:
$$\Gamma G \simeq ({}^a\Gamma{}^aG) \coprod_{({}^c\Gamma {}^cG)} ({}^b\Gamma {}^bG), \quad \text{and} \quad \Gamma \E(G) \simeq {}^a\Gamma\E({}^aG) \coprod_{{}^c\Gamma \E({}^cG)} {}^b\Gamma \E({}^bG).$$
By induction hypothesis and Theorem~\ref{iso finite-case}, the algebras~${}^a\Gamma\E({}^aG)$,  ${}^c\Gamma \E({}^cG)$ and ${}^b\Gamma \E({}^bG)$ are Koszul. Furthemore, from Lemma~\ref{injection of graphs}, the algebra~${}^c\Gamma \E({}^cG)$ is a subalgebra of~${}^a\Gamma \E({}^aG)$ and~${}^b\Gamma \E({}^bG)$. Using~\cite[Lemma~$3.13$]{blumer2024quadratically} from Blumer, we conclude that the algebra~$\Gamma \E(G)$ is Koszul. Thus, using Theorem~\ref{iso finite-case}, we conclude that~$\E(\Gamma G)$ is also Koszul.
\end{proof}

\begin{rema}\label{resol kosz}
Alternatively, if we assume that~$\E({}_iG)$ is Koszul for~$1\leq i \leq k$, we can construct a linear resolution of free graded $\Gamma \E(G)$-free modules.

For~$1\leq i \leq k$, we consider a linear resolution of free graded $\E({}_iG)$-modules:
$$\cdots \to^{{}_id_3} ({}_i\PP_2)\F_p \otimes_{\F_p}\E({}_iG) \to^{{}_id_2} ({}_i\PP_1)\F_p \otimes_{\F_p}\E({}_iG) \to^{{}_id_1} \E(_iG) \to^{{}_i\epsilon} \F_p \to 0,$$
where ${}_i\epsilon$ is the augmentation map and, for every integer $j\geq 0$, the set ${}_i\PP_j$ is a basis of the~$\E({}_iG)$-module~$({}_i\PP_j)\F_p \otimes_{\F_p}\E({}_iG)$.

Let us recall that an~$m$-clique~$(i_1,\dots, i_m)$ in~$\Gamma_m$ is ordered by~$i_1<i_2<\dots <i_m$. We define a complex of free $\Gamma\E(G)$-modules, using the following basis:
$$\PP_n\coloneq \bigcup_{m\in \NN} \left\{ \bigcup_{ (i_1,\dots,i_m)\in \Gamma_m}\bigcup_{n_1+\dots+n_m=n}  ( {}_{i_1}v, \dots, {}_{i_m}v)|\quad {}_{i_j}v\in {}_{i_j}\PP_{n_j}\right\}, $$
where the $n_i$'s can also take value zero, and the derivatives:
\begin{multline*}
    d_n \colon \PP_n\F_p \otimes_{\F_p} \Gamma \E(G)\to \PP_{n-1}\F_p \otimes_{\F_p}\Gamma\E(G); 
    \\d_n(({}_{i_1}v,\dots ,{}_{i_m}v) w)\coloneq \sum_{j=1}^m(-1)^{N_j}\left( {}_{i_1}v,{}_{i_2}v,\dots ,{}_{i_j}d_{n_j}({}_{i_j}v),\dots ,{}_{i_m}v\right)w, 
    \\ \text{with } N_1\coloneq 0, \text{ and } N_j\coloneq n_1+\dots +n_{j-1}.
\end{multline*}
We make here a little abuse of notations:~$({}_{i_1}v,{}_{i_2}v,\dots ,{}_{i_j}d_{n_j}({}_{i_j}v),\dots ,{}_{i_m}v)$ does not denote an element in~$\PP_{n-1}$, but an element in~$\PP_{n-1}\F_p \otimes \Gamma\E(G)$ (precisely in~$\PP_{n-1}\F_p \otimes\E({}_{i_j}G)$). We also note that we have splittings (graded $\F_p$-linear maps)~${}_is_j\colon ({}_i\PP_j)\F_p \otimes_{\F_p} \E({}_iG) \to ({}_i\PP_{j+1})\F_p \otimes_{\F_p}\E({}_iG)$ and~${}_is_{-1}\colon \F_p\hookrightarrow \E({}_iG)$, such that:
$$({}_id_{j+1}) \circ ({}_is_j)+({}_is_{j-1})\circ ({}_id_j)={\rm id_{({}_i\PP_{j})\F_p \otimes \E({}_iG)}}-{}_i\epsilon_j,$$
where ${}_i\epsilon_0\coloneq {}_i\epsilon$ and ${}_i\epsilon_j\coloneq 0$ for $j\geq 1$.

To show that the complex $(\PP_n\F_p\otimes_{\F_p}\Gamma \E(G), d_n)_{n\in \NN}$ is acyclic, we introduce~$s_{-1}\colon \F_p \hookrightarrow \Gamma \E(G)$ and a family of sections~$\{s_j\colon \PP_j\F_p \otimes \Gamma \E(G)\to \PP_{j+1}\F_p\otimes \Gamma \E(G)\}_{j\geq 0}$. For this purpose, we use the Normal Form Theorem from Green~\cite[Theorem~$3.9$]{ruthgraph}. This result was originally provided for groups, but it can be applied for graded algebras. A~$\F_p$-basis of~$\PP_n\F_p \otimes_{\F_p} \Gamma \E(G)$ is given by elements $v_c w$, where: 

$\bullet$ $v_c$ is in $\PP_n$ so $v_c\coloneq ( {}_{i_1}v,\dots ,{}_{i_m}v)$ with $c\coloneq (i_1,\dots,i_m ) \in \Gamma_m$, $n_1+\dots+n_m=n$,

$\bullet$ $w\coloneq ({}_{i_1}v_{j_1})\dots ({}_{i_k}v_{j_k})$,  is a word with letters~${}_iv_j\in \E({}_iG)$ such that for every $l$: 
\\$(i)$ $i_l\neq i_{l+1}$, and $(ii)$ if $i_l>i_{l+1}$, then $\{i_{l}, i_{l+1}\}\notin \bE$.

We first assume that the first letter of $w$ is not in every $_{i_0}V$ such that $( i_0,i_1,\dots, i_m )$ is a $(m+1)$-clique of~$\Gamma$, and~$i_0< i_1 <\dots < i_m$. Therefore, we define:
$$(a)\quad s_n(v_cw)\coloneq \left(  {}_{i_1}s_{n_1}({}_{i_1}v), {}_{i_2}v,\dots ,{}_{i_j}v,\dots ,{}_{i_m}v \right)w, \quad \text{ with } i_1<i_2<\dots <i_m.$$
If we are not in the previous case, then there exists $i_0$ such that $( i_0,\dots,i_m )$ is in $\Gamma_{m+1}$, $i_0<i_1< \dots < i_m$, and the first letter~$w_0$ of $w\coloneq w_0w'$ is in~${}_{i_0}V$. Then we introduce elements ${}_{i_0}v_j$ in ${}_{i_0}\PP_1$ and $w_j$ in $\E({}_{i_0}G)$ such that~${}_{i_0}s_0(w_0)\coloneq \sum_j ({}_{i_0}v_j) w_j$.
Thus we define
$$(b)\quad s_n(v_cw):=\sum_j \left( {}_{i_0}v_j, {}_{i_1}v,\dots, {}_{i_m}v\right) w_jw'.$$
We leave it to the reader to verfify that the preceding maps are indeed splittings.
\end{rema}

\subsection{Cohomology of graph products of Koszul algebras}
Assume that for every~$1\leq i \leq k$, the algebra $\E({}_iG)$ is Koszul. We denote the cohomology of~${}_iG$ by:
$$H^\bullet({}_iG)\simeq \A^\bullet({}_iG)\coloneq {}_i\E/\I^!({}_iG),$$
where $\I^!({}_iG)$ is a two sided ideal of ${}_i\E$ generated by a (known) family that we call~$S^!({}_iG)$ in~${}_i\E_2\subset \E_2$. For background on Koszulity, consult~\cite[Chapters $3$ and $4$]{Loday}.
Let us compute the cohomology of $\Gamma(G)$. We recall that~$\E$ is the set of noncommutative polynomials on~$\{{}_iX_u\}$ with~$1\leq i \leq k$ and~$1\leq u \leq {}_id$ over~$\F_p$, graded by~$\deg({}_iX_u)=1$.
\begin{coro}\label{coho graph}
    We have the following isomorphism of graded algebras:
$$H^\bullet(\Gamma G)\simeq \A^\bullet(\Gamma G)\coloneq\E/\I^!(\Gamma G),$$
where $\I^!(\Gamma G)$ is the two-sided ideal generated by the family:
\begin{multline*}
S^!(\Gamma G):=\bigcup_{i=1}^k S^!({}_iG) \cup S_\Gamma^{(1)}\cup S_\Gamma^{(2)}, \quad \text{where }
\\ S_\Gamma^{(1)}\coloneq \{({}_uX_a)({}_vX_b)+({}_vX_b)({}_uX_a) \text{ for } 1\leq u < v \leq k \text{ and } 1\leq a\leq {}_ud, 1\leq b\leq {}_vd \}
\\S_{\Gamma}^{(2)}\coloneq \{({}_uX_a)({}_vX_b), \text{ for } u\neq v, \{ u,v\} \notin \mathbf{E} \text{ and } 1\leq a\leq {}_ud, 1\leq b\leq {}_vd \}\subset \E_2.
\end{multline*}
%where $(u,v)\not\in \bE$ means either $\{u,v\}\not\in \bE$ or $u>v$.
\end{coro}

\begin{proof}
%In Corollary~\ref{iso enveloppe}, we showed the isomorphism~$\E(\Gamma G)\simeq \Gamma \E(G)$. Thus, from Theorem~\ref{Koszul graph}, 
The algebra~$\E(\Gamma G)$ is Koszul. We compute~$\A^\bullet(\Gamma G)$ the Koszul dual of $\E(\Gamma G)$. Further details on the Koszul duality can be found in~\cite[Chapter $3$, $\S 3.2.2$]{Loday}. For this purpose, we compute~$\I^!(\Gamma G)$, the kernel of the natural surjection:~$\E\to \A^\bullet(\Gamma G)$. We observe that $S^!(\Gamma G)\subset \I^!(\Gamma G)$.
Let us denote by~$S({}_iG)$ (resp.\ $S(\Gamma G)$) a minimal family in~${}_i\E_2$ (resp.\ $\E_2$) generating~$\I({}_iG)$ (resp.\ $\I(\Gamma G)$).
We have
\begin{multline*}
|S(\Gamma G)|={}_1r+\dots +{}_kr+\sum_{1\leq u<v \leq k;\{u,v\} \in \mathbf{E}}{}_ud{}_vd, \quad \text{and} 
\\ |S^!(\Gamma G)|=\left( ({}_1d)^2-{}_1r\right)+\dots +\left( ({}_kd)^2-{}_kr\right)+ \sum_{1\leq v <u \leq k; \{u,v\} \in \mathbf{E}}{}_ud{}_vd+\sum_{1\leq u \neq v \leq k; \{u,v\}\notin \mathbf{E}}{}_ud{}_vd.
\end{multline*}

We finally note that
$$\dim_{\F_p} \E_2=d^2=({}_1d+\dots +{}_kd)^2=|S(\Gamma G)| + |S^!(\Gamma G)|.$$
Thus $S^!(\Gamma G)$ generates the ideal $\I^!(\Gamma G)$. Since~$\E(\Gamma G)$ is Koszul, we conclude using \cite[Proposition~$1$]{Hamza25} (see also~\cite{leoni2024zassenhaus}) that:~$H^\bullet(\Gamma G)\simeq \A^\bullet(\Gamma G).$
\end{proof}

%Let us recall $\P_{{}_iG}(t):=1+\sum_{n\geq 1}h^n({}_1G)t^n$, and we have the relation:
%$$gocha({}_iG,t):=\frac{1}{\P_{{}_iG}(-t)}.$$
%From Corollary \ref{coho graph} and the proof of Theorem \ref{Koszul graph}, we observe that:
%$$h^n(\Gamma G)=\sum_{m\geq 1}\sum_{(i_1,\dots, i_m)\in \Gamma_m}\sum_{n_1+\dots+n_m=n} h^{n_1}({}_{i_1}G)\dots h^{n_m}({}_{i_m}G).$$
Let us now precisely compute the gocha and Poincaré series of $\Gamma G$.

\begin{prop}\label{computation gocha series}
Assume that for every $1\leq i \leq k$, the algebra $\E({}_iG)$ is Koszul, then we have for every positive integer~$n$:
$$h^n(\Gamma G)=\sum_{m\geq 1}\sum_{ (i_1,\dots, i_m) \in \Gamma_m}\sum_{n_1+\dots+n_m=n} h^{n_1}({}_{i_1}G)\dots h^{n_m}({}_{i_m}G).$$
Thus, we infer
$$gocha(\F_p, \Gamma G,t)=\frac{1}{H^\bullet(\Gamma G,-t)} \quad \text{where} \quad H^\bullet(\Gamma G, t)\coloneq 1+\sum_{n\in \NN} h^n(\Gamma G)t^n.$$
\end{prop}

\begin{proof}
We define the series:
$$P_{\Gamma G}(t)\coloneq 1+\sum_{m\geq 1}\sum_{ (i_1,\dots, i_m) \in \Gamma_m} (H^{\bullet}({}_{i_1}G,t)-1)\dots (H^{\bullet}({}_{i_m}G,t)-1)\coloneq 1+\sum_{n\geq 1}p_{\Gamma G,n}t^n.$$
Note that for every integer~$n\geq 1$, we have:
$$p_{\Gamma G,n}\coloneq \sum_{m\geq 1}\sum_{ (i_1,\dots, i_m) \in \Gamma_m} \sum_{n_1+\dots+n_m=n} h^{n_1}({}_{i_1}G)\dots h^{n_m}({}_{i_m}G).$$
We also introduce~$S_{\Gamma G}(t)\coloneq \frac{1}{P_{\Gamma G}(-t)}$. By induction on the number of vertices~$k$ of~$\Gamma$, we show that:
$$P_{\Gamma G}(t)=H^\bullet(\Gamma G, t), \quad \text{and}\quad gocha(\Gamma G,t)=S_{\Gamma G}(t).$$

If~$k=1$, the equalities are true since~$\E(G)$ is Koszul and from~\cite[Proposition~$1$]{Hamza25}.

Now, we assume that~$k\geq 1$. We again distinguish two cases.

$\bullet$ If~$\Gamma$ is a complete graph, then
$$\Gamma G\simeq \prod_{i=1}^k{}_iG \quad \text{and}\quad \E(\Gamma G)\simeq \prod_{i=1}^k\E({}_iG).$$ 
For instance, using~\cite[Proposition~$1$]{Hamza25}, Theorem~\ref{Koszul graph} and~\cite[Chapter~$3$, Corollary~$1.2$]{polishchuk2005quadratic}, we deduce:
$$P_{\Gamma G}(t)=\prod_{i=1}^kH^\bullet({}_iG,t)=H^\bullet(\Gamma G, t), \quad \text{and}\quad gocha(\Gamma G,t)=\prod_{i=1}^k gocha({}_iG,t)=S_{\Gamma G}(t).$$

$\bullet$ Let us assume that~$\Gamma$ is not a complete graph.
Then following the notations and arguments from the proof of Theorem~\ref{Koszul graph}~$(ii)$ we have the following isomorphisms:
$$\Gamma G \simeq ({}^a\Gamma{}^aG) \coprod_{({}^c\Gamma {}^cG)} ({}^b\Gamma {}^bG), \quad \text{and} \quad \Gamma \E(G) \simeq {}^a\Gamma\E({}^aG) \coprod_{{}^c\Gamma \E({}^cG)} {}^b\Gamma \E({}^bG).$$
From~\cite[Lemme~$5.1.10$]{Lemaire} and Theorem~\ref{Koszul graph}, we have:
\begin{multline*}
H^\bullet(\Gamma G,t)=H^\bullet({}^a\Gamma {}^aG,t)+H^\bullet({}^b\Gamma {}^bG,t)-H^\bullet({}^c\Gamma {}^cG,t)=P_{{}^a\Gamma {}^a G}(t)+P_{{}^b\Gamma {}^b G}(t)-P_{{}^c\Gamma {}^c G}(t) 
\\ =1+\sum_{m\geq 1}\sum_{ (i_1,\dots, i_m) \in {}^a\Gamma_m} (H^{\bullet}({}_{i_1}G,t)-1)\dots (H^{\bullet}({}_{i_m}G,t)-1)
\\+\sum_{m\geq 1}\sum_{ (i_1,\dots, i_m) \in {}^b\Gamma_m} (H^{\bullet}({}_{i_1}G,t)-1)\dots (H^{\bullet}({}_{i_m}G,t)-1)
\\-\sum_{m\geq 1}\sum_{ (i_1,\dots, i_m) \in {}^c\Gamma_m} (H^{\bullet}({}_{i_1}G,t)-1)\dots (H^{\bullet}({}_{i_m}G,t)-1)
\\=1+\sum_{m\geq 1}\sum_{ (i_1,\dots, i_m) \in \Gamma_m} (H^{\bullet}({}_{i_1}G,t)-1)\dots (H^{\bullet}({}_{i_m}G,t)-1)=P_{\Gamma G}(t) .
\end{multline*}
From Theorem \ref{Koszul graph}, the algebra $\E(\Gamma G)$ is Koszul. Thus we obtain that~$gocha(\Gamma G,t)=S_{\Gamma G}(t)$, which allows us to conclude.
\end{proof}

\begin{rema}
Alternatively, we can prove Proposition~\ref{computation gocha series} directly from the resolution constructed in Remark~\ref{resol kosz}.
\end{rema}

\begin{coro}
Let $G$ be a family of pro-$p$ groups such that, for every integer~$1\leq i \leq k$, the algebra~$\E({}_iG)$ is Koszul. Then, the cohomological dimension of $\Gamma G$ is
$$ cd( \Gamma G) = \underset{m\in \NN, (i_1,\dots ,i_m) \in \Gamma_m}{\max} \prod_{j=1}^{m} cd({}_{i_j}G).$$
\end{coro}
\begin{proof}
This is just a consequence of Proposition \ref{computation gocha series}.
%he cohomology ring of $\Gamma G$ is generated by elements of $H({}_iG)$ as an algebra with cup product. For a subset $W$ of $V(\Gamma)$ and $x_j \in H^{\bullet}({}_jG)$ for $j \in W$, the product of $x_j$ is trivial unless $\Gamma$ has the complete subgraph with the vertex set $W$ by the anticommutativity of cup product and vanishing.
\end{proof}

\begin{coro}\label{app Koszul}
Assume that for every $1\leq i\leq k$, the algebra $\E(\F_p,{}_iG)$ is Koszul, and the algebra~$\E(\Z_p, {}_iG)$ is torsion-free over~$\Z_p$. Then we infer:
$$gocha(\Z_p,\Gamma G,t)=gocha(\F_p, \Gamma G,t)=\frac{1}{H^\bullet(\Gamma G, -t)},$$
where
$$H^\bullet(\Gamma G, t)\coloneq 1+\sum_n h^n(\Gamma G)t^n, \quad \text{and } h^n(\Gamma G)=\sum_{m\geq 1}\sum_{ (i_1,\dots, i_m) \in \Gamma_m}\sum_{n_1+\dots+n_m=n} h^{n_1}({}_{i_1}G)\dots h^{n_m}({}_{i_m}G).$$
\end{coro}

\begin{proof}
    %$(i)$ If for every $i$, the group ${}_iG$ satisfies the condition $(\ast)$, then the group $\Gamma G$ also satisfies it. Thus, from Theorem \ref{minans}, we deduce that $gocha(\AA_3, \Gamma G)=gocha(\AA_1, \Gamma G)$. Thus the morphism $\E(\AA_1, \Gamma G)\otimes_{\AA_1}\AA_3\simeq \Gamma\E(\AA_1,G)\otimes_{\AA_1}\AA_3 \to \E(\AA_3, \Gamma G)$ is bijective. Consequently, 
    %$$\E(\AA_3,\Gamma G)\simeq \Gamma \E(\AA_3,G).$$

From Proposition~\ref{computation gocha series}, we have 
    $$gocha(\F_p, \Gamma G,t)=\frac{1}{H^\bullet(\Gamma G, -t)}.$$
    We conclude with Theorems~\ref{answerminac} and \ref{answerminac2}.
\end{proof}

%\begin{coro}\label{consminans}
%If
%$\dim_{\F_p}\I_n(\F_p)=\rk_{\Z_p}\I_n(\Z_p)$, then $\E_n(\Z_p,Q)$ is torsion free.
%\end{coro}

%\begin{proof}
%This is a consequence from the proof of Lemma \ref{borninf} and Theorem \ref{minans}.
%\end{proof}

\section{Abstract group}
Observe that the category of finitely generated abstract groups is endowed with a coproduct. Thus, if $G:=\{{}_1G,\dots, {}_kG\}$ is a family of finitely generated abstract groups, we similarly define its graph product $\Gamma G$, which is also finitely generated. Let us fix a prime~$p$ and consider a finitely generated abstract group $Q$. Then we define its pro-$p$ completion~$\widehat{Q}^p$, and~$c_Q$ the natural morphism~$c_Q\colon Q \to \widehat{Q}^p$. Note that~$\widehat{Q}^p$ is a finitely generated pro-$p$ group. From a family~$G$, we infer a family~$\widehat{G}^p:=\{\widehat{{}_1G}^p,\dots , \widehat{{}_kG}^p\}$ of finitely generated pro-$p$ groups. Rather than~$c_{{}_iG}$, we will use the notation~${}_ic\colon {}_iG\to \widehat{{}_iG}^p$. We define $H^\bullet(Q)$ the cohomology ring of the abstract group~$Q$ over~$\F_p$, where~$\F_p$ is endowed with trivial~$Q$-action. We also denote by~$h^n(Q)$, when it is defined, the integer~$\dim_{\F_p} H^n(Q)$, for every positive integer~$n$. Let us recall that we always consider the continuous cohomology for pro-$p$ groups. So~$H^n(\widehat{Q}^p)$ is the continuous~$n$-th cohomology group of the pro-$p$ group~$\widehat{Q}^p$ over~$\F_p$. The group~$Q$ is said~$p$-cohomologically complete if for every positive integer~$n$, we have~$H^n(Q)\simeq H^n(\widehat{Q}^p)$. 
%and~$(b)$ the group~$Q$ is residually~$p$-finite. 

%and we define:
%$$\Gamma G:=\bigast_{i=1}^k {}_iG /\langle [{}_ux_a,{}_vx_b]| \quad \text{for } (u,v)\in \bE\rangle,$$
%where $\bigast$ is the free product on abstract group. 

%Let us observe that $\Gamma G$ can be written as a colimit, we refer to \cite[Example $2.5$]{panovpseries}, thus:
Set~$\mathbb{B}$ either the ring~$\F_p$ or~$\Z$. Let~$A$ and $B$ be subgroups of~$Q$. We denote by $\lbrack A, B\rbrack$ and $AB$ the (abstract) normal subgroups of~$Q$ generated by commutators~$\lbrack a,b\rbrack$ and products~$ab$, with $a$ in $A$ and $b$ in $B$. We also denote by~$A^p$ the normal subgroup of~$Q$ generated by~$a^p$ for~$a$ in~$A$.
We define $Q_\bullet(\Z)$ and $Q_\bullet(\F_p)$ the lower central and the $p$-Zassenhaus series of $Q$ (as an abstract group). Precisely $Q_1(\Z)=Q_1(\F_p):=Q$ and for every positive integer~$n$:
\begin{equation*}
Q_n(\Z):=[Q,Q_{n-1}(\Z)], \quad Q_n(\F_p):=Q_{\lceil \frac{n}{p} \rceil}(\F_p)^p\prod_{i+j=n}[Q_i(\F_p),Q_j(\F_p)].
\end{equation*}

Let~$\Ll(\mathbb{B})$ be the free graded~$\mathbb{B}$-Lie algebra on~$\{X_1,\dots ,X_d\}$. Let~$F$ be an abstract finitely generated free group on~$d$ generators. Similarly to the pro-$p$ case, we have a Magnus isomorphism, which induces an isomorphism:
$$\Ll(\mathbb{B})\simeq \bigoplus_n F_n(\mathbb{B})/F_{n+1}(\mathbb{B}).$$ 
This allows us to introduce~$\Ll(\mathbb{B},Q)$, which is the quotient of~$\Ll(\mathbb{B})$ isomorphic to $\bigoplus_{n\in \NN} Q_n(\mathbb{B})/Q_{n+1}(\mathbb{B})$, through the previous isomorphism. Let us also note that the map~$c_Q\colon Q\to \widehat{Q}^p$ induces two morphisms:
 $$c_Q(\F_p)\colon \Ll(\F_p,Q)\to \Ll(\F_p, \widehat{Q}^p), \quad \text{and} \quad c_Q(\Z_p)\colon \Ll(\Z, Q)\otimes_{\Z}\Z_p \to \Ll(\Z_p, \widehat{Q}^p),$$
satisfying~$c_Q(\AA)(X_i)=X_i$. Observe, when~$Q$ is finitely generated, that~$c_Q(\F_p)$ is an isomorphism.

 This section has two goals. We introduce a first subsection, where we recall tools needed to study pro-$p$ completions of graph products. The second subsection computes the cohomology of~$\Gamma G$, when~$G$ is a family of~$p$-cohomologically complete groups. The last subsection computes~$\Ll(\mathbb{B},\Gamma G)$, under some technical conditions on the family~$G$ related to the pro-$p$ case (precisely~$(\ast_\Z)$).

\subsection{Facts on graph products and their pro-$p$ completions}\label{facts}

We have the following easy results that we will often use later without citing them.

\begin{lemm}\label{completion of graph}
We have an isomorphism 
$\widehat{\Gamma G}^p \simeq \Gamma \widehat{G}^p.$
\end{lemm}

\begin{proof}
We note that~$\widehat{\Gamma G}^p$ and~$\Gamma \widehat{G}^p$ have the same presentation as pro-$p$ groups. Thus they are isomorphic.
\end{proof}

\begin{lemm}\label{p-embed}
Let~${}'\Gamma\coloneq ({}'\bX, {}'\bE)$ be a subgraph of~$\Gamma\coloneq (\bX, \bE)$, and consider~${}'G$ the subfamily of~$G$ indexed by~${}'\bX$. Then we have the following monomorphisms:
$$\iota_{\Gamma/'\Gamma}^{abs} \colon {}'\Gamma ({}'G) \hookrightarrow \Gamma G, \quad \text{and } \iota_{\Gamma/'\Gamma}^p \colon \widehat{{}'\Gamma ({}'G)}^p \hookrightarrow \widehat{\Gamma G}^p.$$
\begin{comment}
Furthermore, assume that, for every~$i\in \bX$, the group~${}_iG$ is residually~$p$-finite. Then we have the following commutative diagram:

\center{
\begin{tikzcd}
{}'\Gamma ({}'G) \arrow[dd, hook] \arrow[rr, "\iota_{\Gamma/'\Gamma}^{abs}", hook]                           &  & \Gamma G \arrow[dd, hook]                       \\
                                                                       &  &                                                 \\
({}'\Gamma)\widehat{{}'G}^p\simeq \widehat{\Gamma ({}'G)}^p \arrow[rr, hook, "\iota_{\Gamma/'\Gamma}^p"] &  & \Gamma \widehat{G}^p\simeq \widehat{\Gamma G}^p
\end{tikzcd}}
\end{comment}
\end{lemm}

\begin{proof}
We define~$\iota_{\Gamma/'\Gamma}^{abs}\colon {}'\Gamma ({}'G) \to \Gamma G$ and~$\iota_{\Gamma/'\Gamma}^p \colon \widehat{{}'\Gamma ({}'G)}^p \hookrightarrow \widehat{\Gamma G}^p$ the canonical maps sending~${}_ix_j$ to~${}_ix_j$ for~$i\in {}'\bX$ and~$1\leq j \leq {}_id$. 

Similarly to the proof of Lemma~\ref{injection of graphs}, we see that~$\iota_{\Gamma/'\Gamma}^{abs}$ and~$\iota_{\Gamma/'\Gamma}^p$ are monomorphisms. 
%Finally the diagram follows from Lemma~\ref{graph p-residual}.
\end{proof}

\subsection{Cohomology on graph products}

The proof of the following result is based on amalgamated products:
\begin{coro}\label{cohomological completness}
We assume that for every~$1\leq i \leq k$, the finitely generated abstract group~${}_iG$ is $p$-cohomologically complete. Then, the group~$\Gamma G$ is~$p$-cohomologically complete.
\end{coro}

\begin{proof}
%From Lemma~\ref{graph p-residual}, if, for every~$1\leq i \leq k$, the group~${}_iG$ is residually~$p$-finite, then~$\Gamma G$ is also. 
We compare the cohomology groups of~$\Gamma G$ and~$\widehat{\Gamma G}^p$. The argument is very similar to the proof of Theorem~\ref{iso finite-case}.

We proceed by induction on~$k$, the number of vertices of~$\Gamma$ (and also the number of groups in the family~$G$). If~$k=1$, then~$G\coloneq \{{}_1G\}$ and~$\Gamma G\coloneq {}_1G$ is~$p$-cohomologically complete. Assume that~$k>1$. We distinguish two cases. 

$\bullet$ If~$\Gamma$ is complete, then~$\Gamma G\simeq \prod_{i=1}^k {}_iG$. Thus from Künneth formula, and its pro-$p$ version given in~\cite[Chapter~II, §4, Exercise~$7$]{NSW}, we show that for every positive integer~$n$:
\begin{multline*}
H^n(\Gamma G)\simeq \bigoplus_{\{j_1,\dots, j_l\}\subset \{1,\dots, k\}} \bigoplus_{n_1+\dots+n_l=n}H^{n_1}({}_{j_1}G)\otimes \dots \otimes                                                     H^{n_{l}}({}_{j_l}G)
\\\simeq \bigoplus_{\{j_1,\dots, j_l\}\subset \{1,\dots, k\}} \bigoplus_{n_1+\dots+n_l =n}H^{n_1}({}_{j_1}\widehat{G}^p)\otimes \dots \otimes                                                     H^{n_l}({}_{j_l}\widehat{G}^p)\simeq H^n(\Gamma \widehat{G}^p).
\end{multline*}

$\bullet$ We assume that~$\Gamma$ is not complete. Then there exists two vertices~$a$ and~$b$ which are not connected by an edge. We define~${}^a\Gamma$, ${}^b\Gamma$ and ${}^c\Gamma$ the subgraphs of~$\Gamma$ with vertices~$\bX\setminus \{a\}$, $\bX\setminus \{b\}$ and~$\bX\setminus \{a,b\}$. We define~${}^aG\coloneq G\setminus \{{}_{a}G\}$,  ${}^bG \coloneq G\setminus \{{}_{b}G\}$ and~${}^cG \coloneq G\setminus \{ {}_{a}G, {}_{b}G\}$. From Lemma~\ref{p-embed} and \cite[Proposition~$9.2.1$]{RibesZal}, we observe (similarly to~\cite[Proof of Lemma~$3.20$]{ruthgraph}) that:
$$\Gamma G\simeq {}^a\Gamma({}^aG) \coprod_{{}^c\Gamma ({}^cG)} {}^b\Gamma ({}^bG), \quad \text{and} \quad \Gamma \widehat{G}^p\simeq ({}^a\Gamma) \widehat{{}^aG}^p \coprod_{({}^c\Gamma) \widehat{{}^cG}^p} ({}^b\Gamma) \widehat{{}^bG}^p.$$
By the induction hypothesis, the groups~${}^a\Gamma({}^aG)$, ${}^b\Gamma({}^bG)$ and ${}^c\Gamma({}^cG)$ are $p$-cohomologically complete. Using Lemma~\ref{p-embed} and~\cite[Corollary~$3.2$]{LORENSEN20106}, we conclude that~$\Gamma G$ is $p$-cohomologically complete.
%we note that~$({}^c\Gamma) \widehat{{}^cG}^p$ (resp.\ $({}^c\Gamma) ({}^cG)$) embeds into~$({}^a\Gamma) \widehat{{}^aG}^p$ and~$({}^b\Gamma)\widehat{{}^bG}^p$ (resp.\ $({}^a\Gamma) ({}^aG)$ and~$({}^b\Gamma)({}^bG)$). Additionally, from Lemma~\ref{p-embed}, the groups~$({}^a\Gamma)\widehat{{}^aG}^p$ and~$({}^b\Gamma)\widehat{{}^bG}^p$ (resp.\ $({}^a\Gamma)({}^aG)$ and~$({}^b\Gamma)({}^bG)$) embed into~$\Gamma \widehat{G}^p$ (resp.\ $\Gamma G$). Using the same argument as~\cite[Corollary~$3.2$]{LORENSEN20106}, we conclude that~$\Gamma G$ is $p$-cohomologically complete.
\end{proof}

\begin{rema}\label{computation cohomology abstract groups}
The author would like to thank Li Cai for bringing the following general result to his attention. We assume that for every~$1\leq i \leq k$ and for every positive integer~$n$ we have $h^n({}_iG)<\infty$. Then, for every positive integer~$n$, we obtain, from \cite[Theorem~$2.35$]{bahri2010polyhedral}, the formula:
$$h^n(\Gamma G)=\sum_{m\geq 1}\sum_{ (i_1,\dots, i_m) \in \Gamma_m}\sum_{n_1+\dots+n_m=n} h^{n_1}({}_{i_1}G)\dots h^{n_m}({}_{i_m}G).$$
From this argument and Proposition~\ref{computation gocha series}, we can deduce an alternative proof of~Corollary~\ref{cohomological completness}, when~$\E(\F_p,\widehat{{}_iG}^p)$ is Koszul, for~$1\leq i \leq k$.
\end{rema}

\begin{exem}
Let us consider the case where~${}_1G={}_2G=\dots ={}_kG=\Z$. We observe that~$\Z$ is~$p$-cohomologically complete for every prime~$p$. Thus RAAGs are~$p$-cohomologically complete for every~$p$. This case was studied by Lorensen using HNN-extensions in~\cite{LORENSEN20106}.
\end{exem}

\subsection{Filtrations on graph products}
For every~$i$, we denote by~${}_ic$ the natural morphism ${}_ic\colon {}_iG\to \widehat{{}_iG}^p$. Observe that~${}_ic$ induces two morphisms:
$${}_ic(\F_p)\colon \Ll(\F_p, {}_iG)\to \Ll(\F_p, \widehat{{}_iG}^p), \quad \text{and} \quad {}_ic(\Z_p)\colon \Ll(\Z, {}_iG)\otimes_{\Z}\Z_p \to \Ll(\Z_p, \widehat{{}_iG}^p).$$
Since~${}_iG$ is finitely generated, then the map~${}_ic(\F_p)$ is an isomorphism. Let us recall that the family $G$ satisfies $(\ast_{\Z})$ if for every $1\leq i\leq k$, and every prime~$p$ we have:

\begin{equation}\tag{$\ast_\Z$}
\left\{ \begin{aligned}
(i) & \text{ the } \Z_p\text{-module } \Ll(\Z_p, \widehat{{}_iG}^p) \text{ is free,}
\\ (ii) & \text{ the map } {}_ic(\Z_p)\colon \Ll(\Z,{}_iG)\otimes_\Z \Z_p \to \Ll(\Z_p, \widehat{{}_iG}^p) \text{ is an isomorphism.}
\end{aligned} \right.
\end{equation} 
%\\ (iii) & \text{ the group } {}_iG \text{ is residually }p\text{-finite.}

In addition, we recall that a family of finitely generated abstract groups~$G$ defines families of graded locally finite~$\AA$ and~$\mathbb{B}$-Lie algebras~$\Ll(\AA, \widehat{G}^p)$ and~$\Ll(\mathbb{B},G)$. These families allow us to define~$\Gamma \Ll(\AA, \widehat{G}^p)$ and~$\Gamma \Ll(\mathbb{B},G)$.
%From the condition~$(iii)$ and~\cite[Chapitre II, I.I.8.3]{lazard1965groupes}, we observe that for every~$1\leq i \leq k$, we have~$\Ll(\F_p,{}_iG)\simeq \Ll(\F_p, \widehat{{}_iG}^p)$. 
This subsection aims to show the following result:

\begin{theo}\label{abstract computations}
Let $p$ be a prime and $G$ be a family of finitely generated abstract groups~$G\coloneq \{{}_1G,\dots, {}_kG\}$. 
%Assume for every~$1\leq i \leq k$ that~${}_iG$ is residually~$p$-finite. 
Then
%~$\Gamma G$ is residually~$p$-finite and:
$$\Ll(\F_p,\Gamma G)\simeq \Gamma \Ll(\F_p,G).$$

Furthermore, if the family $G$ satisfies $(\ast_\Z)$, then
$\Ll(\Z,\Gamma G)$ is torsion-free and 
$$\Ll(\Z, \Gamma G)\simeq \Gamma \Ll(\Z,G), \quad \text{ and } \quad  \Ll(\Z,\Gamma G)\otimes_\Z \Z_p \simeq \Gamma \Ll(\Z_p, \widehat{G}^p).$$
As a consequence, for every integer $n$ and every prime $p$, we have:
$$a_n(\Z,\Gamma G)=a_n(\Z_p,\Gamma \widehat{G}^p).$$
\end{theo}

Let us note, using the Magnus isomorphism (for the abstract case), that~$\Gamma \Ll(\mathbb{B}, G)$ and~$\Ll(\mathbb{B},\Gamma G)$ are both quotients of~$\Ll(\mathbb{B})$, the free-$\mathbb{B}$-graded Lie algebra on~$\{{}_iX_j\}_{1\leq i \leq k, 1\leq j \leq {}_id}$.

\begin{lemm}\label{Gamma epi2}
Assume that~$G$ is a family of finitely generated abstract groups, then we have a canonical surjection:
$$\alpha \colon \Gamma \Ll(\Z, G) \twoheadrightarrow \Ll(\Z, \Gamma G), \quad {}_iX_j \mapsto {}_iX_j,$$
for every~$1\leq i \leq k$ and every~$1\leq j \leq {}_id$. 
\end{lemm}

\begin{proof}
The proof is very similar to the proof of Lemma~\ref{Gamma epi}, using the Magnus isomorphism in the abstract case.
\end{proof}

Our main strategy is to prove, under the condition~$(\ast_\Z)$, that~$\alpha$ is an isomorphism. Before proving Theorem~\ref{abstract computations}, we need the following result:

\begin{lemm}\label{graph product Z}
Assume that the family $G$ satisifes $(\ast_\Z)$. Then $\Gamma \Ll(\Z,G)$ is torsion-free over $\Z$, and for every prime~$p$, we have an isomorphism:
$$\omega_p\colon \Ll(\Z_p, \Gamma \widehat{G}^p) \simeq \left( \Gamma \Ll(\Z,G)\right)\otimes_\Z \Z_p, \quad {}_iX_j \mapsto {}_iX_j,$$
for every~$1\leq i \leq k$ and every~$1\leq j \leq {}_id$.
\end{lemm}

\begin{proof}
Since $\Gamma \Ll(\Z,G)$ is a colimit (see for instance \cite[Section $2$]{panovpseries}), we first have the isomorphism:
$$\left( \Gamma \Ll(\Z,G)\right)\otimes_\Z \Z_p \simeq \Gamma \left( \Ll(\Z,G)\otimes_\Z \Z_p\right),$$
where
$$\Ll(\Z,G)\otimes_\Z \Z_p\coloneq \{\Ll(\Z, {}_1G)\otimes_\Z \Z_p, \dots, \Ll(\Z, {}_kG)\otimes_\Z \Z_p\}\simeq \Ll(\Z_p, \widehat{G}^p),$$
and the last isomorphism is given by~$(\ast_\Z, ii)$.

Observe that $\Gamma \Ll(\Z,G)$ is torsion-free over~$\Z$, if and only if for every prime $p$, the~$\Z_p$-module~$\left( \Gamma \Ll(\Z,G)\right)\otimes_\Z \Z_p$ is torsion-free. We infer from the previous isomorphism and Theorem~\ref{main result} the following isomorphisms, defining~$\omega_p$:
 $$\Ll(\Z_p, \Gamma \widehat{G}^p)\simeq \Gamma \Ll(\Z_p, \widehat{G}^p)\simeq \Gamma \left( \Ll(\Z,G)\otimes_\Z \Z_p\right) \simeq  \left( \Gamma \Ll(\Z,G)\right)\otimes_\Z \Z_p.$$
Furthermore, from Theorem~\ref{main result}, the module $\Ll(\Z_p, \Gamma \widehat{G}^p)$ is torsion-free for every prime~$p$. Thus $\Ll(\Z,\Gamma G)$ is also torsion-free.
\end{proof}

We now prove Theorem~\ref{abstract computations}.

\begin{proof}[Proof Theorem \ref{abstract computations}]
%From Lemma~\ref{graph p-residual}, the group~$\Gamma G$ is residually~$p$-finite. 

We distinguish several cases:

$\bullet$ We start to show that $\Ll(\F_p, \Gamma G)\simeq \Gamma \Ll(\F_p,G)$ for a fixed prime $p$. We recall that, for every finitely generated abstract group $Q$, 
%that $\widehat{Q}^p$ is the completion of $Q$ with the topology given by $\{Q_n(\F_p)\}_{n\in \NN}$. As a consequence 
 we have $\Ll(\F_p,Q)\simeq \Ll(\F_p, \widehat{Q}^p)$.  %See~\cite[Chapitre II, I.I.8.3]{lazard1965groupes}. From Lemma~\ref{graph p-residual}, we infer that~$\Gamma G$ is residually~$p$-finite. 
 Thus, from Theorem \ref{main result}, we have
$$\Ll(\F_p, \Gamma G)\simeq \Ll(\F_p, \widehat{\Gamma G}^p)\simeq \Ll(\F_p, \Gamma \widehat{G}^p)\simeq \Gamma \Ll(\F_p, \widehat{G}^p)\simeq \Gamma \Ll(\F_p, G).$$

$\bullet$ We now assume that~$G$ satisifes~$(\ast_\Z)$. We show that~$\alpha\colon \Gamma \Ll(\Z,G)\to \Ll(\Z, \Gamma G)$, defined by Lemma~\ref{Gamma epi2}, is an isomorphism.
For this purpose, we show that for every prime~$p$, the associated morphism:
$$\alpha_p\colon \left( \Gamma  \Ll(\Z,G)\right) \otimes_\Z \Z_p \to \Ll(\Z, \Gamma G)\otimes_\Z \Z_p$$
is an isomorphism. Since $\alpha$ is surjective, then $\alpha_p$ also and we notice that $\alpha_p$ maps~${}_iX_j$ to~${}_iX_j$ where $1\leq i\leq k$ and $1\leq j \leq {}_id$.
Recall that the morphism~$c_{\Gamma G}\colon \Gamma G \to \widehat{\Gamma G}^p$ defines a morphism:
$$c_{\Gamma G}(\Z_p) \colon \Ll(\Z,\Gamma G)\otimes_\Z \Z_p \to \Ll(\Z_p, \Gamma \widehat{G}^p);\quad {}_iX_j \mapsto {}_iX_j,$$
where $1\leq i \leq k$ and $1\leq j \leq {}_id$. 
From Lemma \ref{graph product Z}, we deduce the following chain of morphisms: 
%where $u_p\colon \Ll(\Z_p, \Gamma \widehat{G}^p)$ is the injection given by PBW Theorem (since from Theorem~\ref{main result}, the module $\Ll(\Z_p, \Gamma \widehat{G}^p)$ is torsion free):

\centering{% https://tikzcd.yichuanshen.de/#N4Igdg9gJgpgziAXAbVABwnAlgFyxMJZABgBoBGAXVJADcBDAGwFcYkQAdDgGUYAouALQD6aUlwDi9ALbT6AAi4B3LLAAW9HMAkBfAHpoAlCB2l0mXPkIoATKWLU6TVu0ky5invyGkJhrhB40vDCQp4iaCZmIBjYeAREACz2jgwsbIicXgIcguIcUrIKfgFBIUKixqbmcVZEAMwUqc4ZWbw5EfmFHsqqMBpaugZVjjBQAObwRKAAZgBOENJIZCA4EEjkNDj0WIzsahAQANYgNGow9FDsOEoQ55cINGkumQHB4-SiUbMLS4ibq3WiDsIHuV0yNzuFygjyc6VcHCYaA0X2qIHmiyQILWSEacJeWSRKLcRW+6N+yy2QPqOkoOiAA
\begin{tikzcd}
                                                                                      &  & {\Gamma \Ll(\Z,G)\otimes_\Z \Z_p} \arrow[rr, "\alpha_p", two heads] &                                  & {\Ll(\Z,\Gamma G)\otimes_\Z \Z_p} \arrow[ld, "c_{\Gamma G}(\Z_p)"] \\
{\Ll(\Z_p,\Gamma \widehat{G}^p)} \arrow[rru, "\omega_p", two heads, hook] \arrow[rrr, "c_{\Gamma G}(\Z_p) \circ \alpha_p\circ \omega_p"] &  &                                                                     & {\Ll(\Z_p,\Gamma \widehat{G}^p)} &                                                            
\end{tikzcd}}

\justifying
For every $1\leq i\leq k$ and every $1\leq j \leq {}_id$, we have~$c_{\Gamma G}(\Z_p) \circ \alpha_p\circ \omega_p({}_iX_j)={}_iX_j$. Thus we deduce that $c_{\Gamma G}(\Z_p) \circ \alpha_p\circ \omega_p$ is the identity, and so an isomorphism. Since~$\omega_p$ is an isomorphism and~$\alpha_p$ is surjective, we deduce that~$\alpha_p$ is injective, so also an isomorphism. Thus~$c_{\Gamma G}(\Z_p)$ is an isomorphism. This allows us to conclude.
\end{proof}

\begin{rema}[Group algebras and filtrations]
Let~$p$ be a prime. Assume that~$Q$ is a finitely generated abstract group. Consider~$E(\mathbb{B},Q)$ the group algebra of~$Q$ over~$\mathbb{B}$. Let~$n$ be a positive integer, and denote by~$E_n(\mathbb{B},Q)$ the~$n$-th power of the augmentation ideal of~$E(\mathbb{B},Q)$. We define~$\E(\mathbb{B},Q)\coloneq \bigoplus_{n\in \NN} E_n(\mathbb{B},Q)/E_{n+1}(\mathbb{B},Q)$. We observe that~$\E(\F_p,Q)\simeq \E(\F_p, \widehat{Q}^p)$.

Let us also note that if $G$ is a family of abstract groups satisfying $(\ast_\Z)$, then~$\Ll(\Z,\Gamma G)$ is torsion-free over~$\Z$. Moreover its universal envelope is given by~$\E(\Z, \Gamma G)$. For further details, consult~\cite[Theorem~$1.3$]{hartlfox}. 
\end{rema}
\bibliography{bib}
\bibliographystyle{plain}
\end{document}